\documentclass[10pt]{amsart}
\usepackage{amsfonts}
\usepackage{amssymb}
\usepackage{graphicx,color}

\usepackage{a4wide}

\usepackage{amsmath}

\usepackage{amsthm}

\usepackage{color, colortbl}

\definecolor{Gray}{gray}{0.9}
\definecolor{lgrey}{RGB}{243,243,243}
\definecolor{lred}{RGB}{185,0,0}
\definecolor{lblue}{RGB}{1,17,207}

\DeclareGraphicsExtensions{.pdf,.jpg,.png,.eps,.gif}
\title{Bounds for spherical codes: the Levenshtein framework lifted}
\date{\today}

\newtheorem{theorem}{Theorem}[section]
\newtheorem{lemma}[theorem]{Lemma}
\newtheorem{corollary}[theorem]{Corollary}

\newtheorem{conjecture}[theorem]{Conjecture}

\theoremstyle{definition}
\newtheorem{definition}[theorem]{Definition}

\newtheorem{remark}[theorem]{Remark}

\newcommand{\Sp}{\mathbb{S}}
\newcommand{\R}{\mathbb{R}}

\begin{document}

\author{P. G. Boyvalenkov}\thanks{$^1$The research of the first and fifth authors was supported, in part, by a Bulgarian NSF contract DN02/2-2016.}
\address{ Institute of Mathematics and Informatics, Bulgarian Academy of Sciences, \\
8 G Bonchev Str., 1113  Sofia, Bulgaria
and Technical Faculty, South-Western University, Blagoevgrad, Bulgaria }
\email{peter@math.bas.bg}

\author{P. D. Dragnev}\thanks{$^2$The research of the second author was supported, in part, by a Simons Foundation grant no. 282207.}
\address{ Department of Mathematical Sciences, Purdue University \\
Fort Wayne, IN 46805, USA }
\email{dragnevp@pfw.edu}

\author{D. P. Hardin}\thanks{$^3$The research of the third and fourth authors was supported, in part,
by the U. S. National Science Foundation under grant DMS-1516400.}
\address{ Center for Constructive Approximation, Department of Mathematics, \\
Vanderbilt University, Nashville, TN 37240, USA }
\email{doug.hardin@vanderbilt.edu}

\author{E. B. Saff}
\address{ Center for Constructive Approximation, Department of Mathematics, \\
Vanderbilt University, Nashville, TN 37240, USA }
\email{edward.b.saff@vanderbilt.edu}

\author{M. M. Stoyanova}
\address{ Faculty of Mathematics and Informatics, Sofia University, \\
5 James Bourchier Blvd., 1164 Sofia, Bulgaria}
\email{stoyanova@fmi.uni-sofia.bg}

\maketitle

\vspace*{-10mm}

\begin{abstract}
Based on the Delsarte-Yudin linear programming approach, we extend Levenshtein's framework to obtain lower bounds for the minimum $h$-energy of spherical codes of prescribed dimension and cardinality,
and upper bounds on the maximal cardinality of spherical codes of prescribed dimension and minimum separation. 
These bounds are universal in the sense that they hold for a large class of potentials $h$ and in the sense of Levenshtein. Moreover, codes attaining the bounds are universally optimal in the sense of Cohn-Kumar. 
Referring to Levenshtein bounds and the energy bounds of the authors as ``first level", our results can be considered as ``next level" universal bounds as they have the same general nature and imply necessary and sufficient conditions for their local and global optimality. For this purpose, we introduce the notion of Universal Lower Bound space (ULB-space), a space that satisfies certain quadrature and interpolation properties.
While there are numerous cases for which our method applies, we will emphasize the model examples of $24$ points ($24$-cell) and $120$ points ($600$-cell) on $\mathbb{S}^3$.
In particular, we provide a new proof that the $600$-cell is universally optimal, and in so doing, we derive optimality of the $600$-cell on a class larger than the absolutely monotone potentials considered by Cohn-Kumar. 
\end{abstract}

\


\


\

\maketitle

\section{Introduction}

Let $\Sp^{n-1}$ denote the unit sphere in $\R^n$. We consider finite configurations (codes) $C \subset \Sp^{n-1}$ of $N \geq 2$ points.
Given an (extended real-valued) function $h(t):[-1,1] \to [0,+\infty]$, the {\em $h$-energy} of $C$ is given by
\[ E_{h}(C):=\sum_{x, y \in C, x \neq y} h(\langle x,y \rangle), \]
where $\langle x,y \rangle$ denotes the usual inner product of $x$ and $y$.
We are interested in lower bounds on the {\em minimal energy}
\begin{equation}\label{Energy}
 \mathcal{E}_h(n,N):=\inf_{|C|=N} \{E_h(C) \},
\end{equation}
where $|C|$ denotes the cardinality of $C$.

Delsarte-Yudin's approach \cite{Y,CK,BDHSS} for finding such lower bounds by linear programming (LP) is described as follows.
Let $A_{n,h}$ denote the {\em feasible domain} of continuous functions $f$ on $[-1,1]$:
\begin{equation}\label{fform}
A_{n,h}:=\{ f\, : \, f(t)=\sum_{i=0}^\infty f_i P_i^{(n)}(t)\leq h(t), \ t \in [-1,1], \ f_i\ge 0,\  i= 1,2,\dots\},
\end{equation}  where $\{P_i^{(n)}(t)\}$ are the Gegenbauer polynomials \cite{Sze} normalized by $P_i^{(n)}(1)=1$ (see Section 2.1).
Then \eqref{Energy} can be estimated by elements of the class \eqref{fform} by
\begin{equation}\label{glowbound2}
 \mathcal{E}_h(n,N)\ge \max_{f\in A_{n,h}} N^2\left(f_0  -\frac{f(1)}{N}\right).
\end{equation}

Hereafter we consider only absolutely monotone potentials $h$.

\begin{definition}
An extended real-valued function $h(t):[-1,1] \to (0,+\infty]$ is called \emph{absolutely monotone} if
$h^{(i)}(t)\geq 0$, for every $t\in[-1,1)$ and every integer $i \geq 0$, and $h(1)=\lim_{t \to 1^{-}} h(t)$.
\end{definition}

Instead of solving the infinite-dimensional linear program on the right-hand side of \eqref{glowbound2} one
may restrict to a subspace $\Lambda \subset C([-1,1])$ (usually finite-dimensional), namely determining the quantity
\begin{equation}\label{Wdef} \mathcal{W}_{h,\Lambda}(n,N):= \sup_{f \in \Lambda \cap  A_{n,h}} N^2\left(f_0  -\frac{f(1)}{N}\right). \end{equation}

\begin{definition}
\label{opti-poly}
A polynomial $f \in \Lambda \cap A_{n,h}$ is called $\Lambda$-LP-\emph{optimal}, if it attains the supremum in \eqref{Wdef}.
If $\Lambda$ is the space of all real polynomials, then $f$ is called {\em LP-optimal}.
\end{definition}

In our considerations the space $\Lambda$ depends on $n$ and $N$. We are interested in a special class of spaces that we call \emph{ULB-spaces}
(see subsection \ref{ULB-space-def} for the definition).

In \cite{BDHSS} we derived {\em Universal Lower Bounds} (ULB) on energy by
finding $\mathcal{W}_{h,\Lambda}(n,N)$ when $\Lambda= \mathcal{P}_m$, the space of
polynomials of degree at most $m=\tau(n,N)$, where $\tau(n,N)$ is the unique positive integer such that the cardinality $N$ lies in an interval
defined by consecutive Delsarte-Goethals-Seidel bounds (see \eqref{tauNn} for the rigouros defintion). We found $\mathcal{P}_m$-LP-optimal polynomials
(see \cite[Theorem 3.1]{BDHSS}, as well as Theorem \ref{thm3.2} in this paper). We now call this ULB a \emph{first level ULB} and proceed further with next levels requiring, of course, higher degrees. 
One of the main goals of this article is to introduce a framework for finding a second level $\mathcal{W}_{h,\Lambda}(n,N)$,
where $\Lambda=\mathcal{P}_{\tau(n,N)+4}$. We explain in Section 3 how $\Lambda$-LP-optimal polynomials of degree $\tau(n,N)+4$
can be found in $\Lambda$.

We are also interested in the closely related problem of determining the maximal possible cardinality of
a spherical code $C$ on $\mathbb{S}^{n-1}$ of prescribed maximal inner product $s(C):=\max\{ \langle x,y \rangle \, : \,   x,y \in C,  x\neq y \}$
\begin{equation}
\label{ANS}
\mathcal{A}(n,s):=\max\{|C| \colon C \subset \mathbb{S}^{n-1}, s(C)=s\},
\end{equation}
and the related quantity (also known as the {\em Tammes problem})
\begin{equation}
\label{snm}
s(n,N):=\min \{ s(C) \, : \, C \subset \mathbb{S}^{n-1}, \ |C|=N \}.
\end{equation}

Linear programming upper bounds for $\mathcal{A}(n,s)$ are obtained by polynomials from the class
\begin{equation}\label{gform}
B_{n,s}:=\{ g\, : \, g(t)=\sum_{k=0}^\infty g_k P_k^{(n)}(t)\leq 0, \ t \in [-1,s], \ g_k\ge 0,\  k= 1,2,\dots\},
\end{equation}
(see Theorem \ref{thm-lp-sph} below). The first level upper bound $\mathcal{A}(n,s) \leq L(n,s)$, due to Levenshtein \cite{Lev78,Lev83,Lev92} (see also \cite{DL,KL}),
is obtained by polynomials of degree $m=\tau(n,N)$ that belong to the class \eqref{gform}, as explained in Section 2.4 (see \eqref{lev-poly}).
It is important to note an intimate connection between $\mathcal{P}_m$-LP-optimal polynomials that yield our ULB and the Levenshtein polynomials, namely that the first interpolate the potential-interaction function at the zeros of the second. The commonality is given by a special $1/N$-quadrature rule with nodes at these zeros (see Definition \ref{1/N-quadrature}). Even though in most of our applications $N$ denotes the cardinality of a code $C$, it is beneficial not to restrict $N$ to positive integers, but to positive reals.

Our results yield second level LP bounds for  $\mathcal{A}(n,s)$ with polynomials of degree $m+4$ for certain $s$ and  $m=\tau(n,L(n,s))$.
The second level bounds for $\mathcal{E}_h(n,N)$ and $\mathcal{A}(n,s)$, where $N=L(n,s)$,
happen simultaneously and exactly when $\Lambda$ is a ULB-space. An important $1/N$-quadrature formula serves as the intersecting aspect again,
and, exactly as at the first level, its nodes are both roots of the improving polynomials
for $ \mathcal{A}(n,s)$ as well as interpolation nodes (to the potential-interaction function) of the $\Lambda$-LP-optimal polynomials for improving the ULB.
Thus, our approach can be viewed as lifting the Levenshtein framework to next level(s).


The paper is organized as follows. In Section 2, following a brief discussion of Gegenbauer polynomials, we define in subsection 2.2 what is meant by a $1/N$-quadrature rule. In subsection 2.3 we introduce the notion of a ULB space.
In subsection 2.4 we explain results of Levenshtein and Delsarte, Goethals, and Seidel
that are instrumental in defining the first level bounds.
In Subsection 2.5 we show how the notion of ULB space unifies our energy bound in \cite{BDHSS} and the Levenshtein bound \cite{Lev} (see Theorem  \ref{LevULBCombined}). The Hermite interpolation framework utilized in ULB space analysis is established in Subsection 2.6. We define the first level ULB
for $\mathcal{E}_h(n,N)$ in Theorem \ref{thm3.2}.
Section 3 examines the second level ULB on energy and second-level Levenshtein-type upper bounds on cardinality of codes with fixed maximum inner product. We consider necessary and sufficient conditions on the so-called  skip two/add two subspace to be ULB space, with the case when $n$ and $N$ are such that $\tau(n,N)=2k-1$ being considered in more detail. Subsection 3.1 is devoted to a detailed examination on necessary and sufficient conditions for existence and uniqueness of a $1/N$-quadrature rule on the skip two/add two subspace. Section 3.2 considers the existence of a Hermite interpolant on the subspace in question. The second level bounds are presented in subsection 3.3 and in subsection 3.4 we briefly review the case  $\tau(n,N)=2k$.  Section 4 demonstrates the second level lift with two model examples, $(n,N)=(4,24)$ and
$(4,120)$, respectively. In the case of $24$ points on $\mathbb{S}^3$ we show that the optimal polynomial
solving \eqref{Wdef} for $\Lambda=\mathcal{P}_9$ is
also a solution of \eqref{glowbound2}. In Section 5 we perform a third level lift of the Levenshtein framework and as a result derive a new proof that the $600$-cell is universally optimal. Moreover, we completely characterize the optimal polynomials of degree at most $17$ for the Delsarte-Yudin linear programming lower bounds by finding two new polynomials that, together with Cohn-Kumar's polynomial, form the vertices of the convex hull that consists of all optimal polynomials. Our framework provides a conceptual explanation of why polynomials of degree $17$ are needed to handle the $600$-cell via linear programming. Section 6 presents numerical results, such as graphics illustrating our new bounds, as well as a sample of an extensive list of cases where the lift of the Levenshtein framework is achievable. We display the improved energy bound and the associated separation distance. More comprehensive list may be found on https://my.vanderbilt.edu/edsaff/.

\section{Preliminaries}

\subsection{Gegenbauer and adjacent polynomials}

We denote by $P_i^{(n)}(t)$ the Gegenbauer polynomials of respective degrees $i$ that are
orthogonal with respect to the measure
\[d\mu(t):=\gamma_n (1-t^2)^{\frac{n-3}{2}}\, dt, \quad t\in [-1,1], \]
where
$ \gamma_n := \Gamma(\frac{n}{2})/\sqrt{\pi}\Gamma(\frac{n-1}{2})$
is a normalizing constant that makes $\mu$ a probability measure. We impose the
normalization $P_i^{(n)}(1)=1$ for every $i$, and note that $P_i^{(n)}(t)$ is just the Jacobi
polynomial $P_i^{(\frac{n-3}{2},\frac{n-3}{2})}(t)$.

For $a,b \in \{0,1\}$ and $i \geq 0$, we denote by $P_i^{a,b}(t):=
P_i^{(a+\frac{n-3}{2},b+\frac{n-3}{2})}(t)$ the corresponding Jacobi polynomial, normalized
again by $P_i^{a,b}(1)=1$, which is called an {\sl adjacent polynomial} \cite{Lev92}.
The (probability) measure of orthogonality for the sequence of these polynomials is 
\begin{equation}\label{nuab}
d\nu^{a,b}(t):=c^{a,b} (1-t)^a(1+t)^b d\mu(t),
\end{equation}
where $c^{a,b}$ is a normalizing constant. Of course, if $a=b=0$ we
have the Gegenbauer polynomials where we use instead the $(n)$ indexing.
Connections between the Gegenbauer polynomials and their adjacent polynomials are given by the Christoffel-Darboux
formulas (see \cite[Lemma 5.24]{Lev}).

We write $\langle f,g \rangle_{a,b}$ to denote the inner product
\[ \langle f,g \rangle_{a,b}:=\int_{-1}^1 f(t)g(t) \, d\nu^{a,b}(t), \]
and set $\| f \|_{a,b}^2 := \langle f,f \rangle_{a,b}$,  $\| f \|^2 := \langle f,f \rangle$, where
$ \langle \cdot,\cdot \rangle := \langle \cdot,\cdot \rangle_{0,0}$.

For any real polynomial $f(t)$ of degree $r$ we have for fixed $a,b$ the unique expansion
\[ f(t)= \sum_{i=0}^r f_i^{a,b} P_i^{a,b}(t) \]
with well-defined coefficients
\[ f_i^{a,b} := r_i^{a,b} \int_{-1}^1 f(t)P_i^{a,b}(t) \, d\nu^{a,b}(t), \]
where
\begin{eqnarray*}
r_i^{a,b} &:=& \left(\int_{-1}^1 \left(P_i^{a,b}(t)\right)^2 \, d\nu^{a,b} (t) \right)^{-1}=\|P_i^{a,b}\|_{a,b}^{-2} \\
\end{eqnarray*}
(see Lemma 5.24 in \cite{Lev}) and
\[r_i=r_i^{0,0}=\frac{2i+n-2}{i+n-2} {i+n-2 \choose i}.  \]
For future reference, we write
\[ P_i^{a,b}(t)=\sum_{j=0}^i a_{i,j}^{a,b} t^j. \]

\begin{definition}
A polynomial $f(t)$ of degree $r$ is called ({\it strictly}) {\it positive definite} if $f_i=f_i^{0,0} \geq 0$ for every
$i$ (if $f_i>0$ for every $i=0,1,\ldots,r$) in its
Gegenbauer expansion.
\end{definition}

Similarly, we consider $(1,0)$-positive definite polynomials (note that $(1,0)$-positive definiteness implies
positive definiteness since by the Christoffel-Darboux formula  $P_i^{1,0}(t)=\sum_{j=0}^i r_j P_j^{(n)}(t)/ \sum_{j=0}^i r_j $ with
positive $r_j=r_j^{0,0}$).

We also employ the so-called {\em Krein condition} and in some cases the {\em strengthened Krein condition} (see \cite[Section 3.3]{Lev}). Both conditions
are satisfied by the Gegenbauer polynomials and their adjacent polynomials since
\begin{equation}
\label{krein}
P_i^{(n)}(t)P_j^{(n)}(t)
\end{equation}
and
\begin{equation}
\label{s-krein}
(t+1)P_i^{1,1}(t)P_j^{1,1}(t)
\end{equation}
are respectively positive definite and strictly positive definite for any nonnegative integers $i,j$ (see  \cite{Gas}, \cite[Lemma 3.22]{Lev}).

\subsection{$1/N$-Quadrature rules and  lower bounds for energy on subspaces}

An important ingredient in obtaining LP bounds for $\mathcal{E}_h(n,N)$ and $\mathcal{A}(n,s)$
is the notion of a $1/N$-quadrature rule over subspaces introduced by the authors in \cite{BDHSS}, which we briefly review here.

\begin{definition} \label{1/N-quadrature} (\cite{BDHSS}, Section 2.2) A finite sequence of ordered pairs
$\{(\alpha_i, \rho_i)\}_{i=1}^{k}$, $-1 \leq \alpha_1 < \alpha_2 <\cdots <\alpha_{k} < 1$,
$\rho_i>0$ for $i=1,2,\ldots,k$, forms a {\em $1/N$-quadrature rule that is exact  for a subspace  $\Lambda \subset C[-1,1]$}  if
\[ f_0=\int_{-1}^1f(t)d\mu(t)= \frac{f(1)}{N}+ \sum_{i=1}^{k} \rho_i f(\alpha_i), \]
 for all $f\in \Lambda$.
\end{definition}

The following results from \cite{BDHSS} find frequent utilization in the present work.

\begin{theorem}[\cite{BDHSS}, Theorems 2.3 and 2.6]\label{THM_subspace}
Let $\{(\alpha_i, \rho_i)\}_{i=1}^{k}$ be a $1/N$-quadrature rule that is exact
for a subspace $\Lambda\subset C([-1,1])$. If $f\in \Lambda\cap A_{n,h}$, then
$\mathcal{E}_h(n,N) \ge  N^2\sum_{i=1}^{k} \rho_i f(\alpha_i)$ and
\begin{equation}
\label{Wineq1}
\mathcal{W}_{h,\Lambda}(n,N)\le N^2\sum_{i=1}^{k} \rho_i h(\alpha_i).
\end{equation}
If there exists an $f\in \Lambda \cap  A_{n,h}$ such that $f(\alpha_i)=h(\alpha_i)$ for $i=1,\ldots, k$, then
equality holds in \eqref{Wineq1} (i.e., $f$ is $\Lambda$-LP-optimal) which yields the   lower bound
\begin{equation}
\label{ULB1}
\mathcal{E}_h(n,N) \ge  N^2\sum_{i=1}^{k} \rho_i h(\alpha_i).
\end{equation}
Furthermore, in this case if $\Lambda' :=\Lambda \oplus \text{span}\{P_j^{(n)} \colon  j\in \mathcal{I}\}$ for some index set $\mathcal{I}\subseteq {\mathbb N}$ and the quantities
\begin{equation}
\label{TestFunctions}
Q_j^{(n,N)}:= \frac{1}{N}+\sum_{i=1}^{k}\rho_i P_j^{(n)}(\alpha_i)
\end{equation}
satisfy $Q_j^{(n,N)} \ge 0 $ for $j\in \mathcal{I}$, then
\[ \mathcal{W}_{h,\Lambda^\prime}(n,N)=\mathcal{W}_{h,\Lambda}(n,N)=N^2\sum_{i=1}^{k}\rho_i h(\alpha_i); \]
i.e., $f$ is $\Lambda^\prime$-LP-optimal. In particular, if $\mathcal{I}=\mathbb{N}$, then $f$ is LP-optimal.
\end{theorem}

%

The quantities $Q_j^{(n,N)}$ are called {\em test functions}\footnote{In fact, they coincide with the test functions from \cite{BDB} 
used for investigations of maximal spherical codes (see also \cite[Theorem 5.47]{Lev}). The name "test functions" reflects their main property as an indicator whether a corresponding universal bound is LP-optimal or
could be improved by linear programming.}  and were introduced and investigated in \cite{BDHSS} for the Levenshtein quadrature (see \eqref{quadrature_nodes}-\eqref{defin_qf}) when $\Lambda={\mathcal P}_m$, $m\leq \tau(n,N)$, where $\tau(n,N)$ is defined in \eqref{tauNn} below. It follows from the Levenshtein quadrature (22) applied for the polynomials $P_j^{(n)}(t)$, $j=1,\dots,\tau(n,N)$, that $Q_j^{(n,N)}=0$, $j=1,\dots,\tau(n,N)$. It was shown  in \cite{BDB} that both $Q_{\tau(n,N)+1}^{(n, N)}\ge 0$ and  $Q_{\tau(n,N)+2}^{(n,N)}\ge 0$ resulting in $\mathcal{W}_{h,\mathcal{P}_{\tau(n,N)}}(n,N)=\mathcal{W}_{h,\mathcal{P}_{\tau(n,N)+2}}(n,N)$. Moreover, a conjecture \cite[Conjecture 4.8]{BDHSS} tested on numerous occasions states that if $Q_{\tau(n,N)+3}^{(n,N)}\geq 0$ and $Q_{\tau(n,N)+4}^{(n)}\geq 0$, then $Q_j^{(n,N)}\geq 0$ for all $j$ and hence, the ULB is LP-optimal.

Motivated by this we consider in this article cases when at least one of the test functions $Q_{\tau(n,N)+3}^{(n,N)}$ and $Q_{\tau(n,N)+4}^{(n, N)}$ is strictly negative and consider solving $\mathcal{W}_{h,\Lambda(n,N)}$ for {\em skip-two add-two} subspaces of 
\begin{equation}\label{Lambdakn} \Lambda_{n,k}:=\mathcal{P}_{\tau(n,N)}\oplus \text{ span}\{P^{(n)}_{\tau(n,N)+3},P^{(n)}_{\tau(n,N)+4}\},\end{equation}
where $k:=\lceil\tau(n,N)/2\rceil$; i.e.,  $\tau(n,N)=2k-1$ or $2k$.  In fact, it was also shown in \cite{BDB} that
for each dimension $n$, there is some $k_0=k_0(n)$ such that $Q_{2k+3}^{(n,N)}<0$ for $k\ge k_0$, which yields that for fixed $n$, if the cardinality $N$ is sufficiently large, we have $\mathcal{W}_{h,\Lambda_{n,k}}(n,N)>\mathcal{W}_{h,\mathcal{P}_{\tau(n,N)}}(n,N)$.   In this paper we develop necessary and sufficient conditions for the existence of a $1/N$-quadrature rule that is exact on $\Lambda_{n,k}$ and such that there is an $f\in \Lambda_{n,k}\cap A_{n,h}$ that interpolates $h$ at the nodes of the quadrature for any absolutely monotone $h$, and thus, provides an improved (or second level) ULB in these cases.  

%
%

\subsection{ULB-spaces}
\label{ULB-space-def} The above theorem motivates the following definition.

\begin{definition}
\label{def-ULB-space}
Let $n$ and $N$ be positive integers. A space $\Lambda\subset C[-1,1]$   is called a
{\em ULB-space} for dimension $n$ and cardinality $N$ and \eqref{ULB1} is called a {\em universal lower bound (ULB)}  if the following two conditions hold:

\rm{(i)} there exists a $1/N$-quadrature rule that is exact for $\Lambda$ (see Definition  \ref{1/N-quadrature});

\rm{(ii)} for any absolutely monotone function $h$ there exists some  $f\in \Lambda\cap A_{n,h}$ that agrees with $h$ at the nodes of the $1/N$-quadrature rule from (i).
\end{definition}

Our first result presents a lower bound on the quantity \eqref{snm}, which follows easily by applying the definition above.
\begin{theorem}\label{LevUpper}
Suppose $\Lambda$ is a ULB-space for given $n$ and $N$, and let $\alpha_k$ denote the largest node less than $1$ in
the $1/N$-quadrature rule.
Then $s(n,N)\geq \alpha_k$.
\end{theorem}
\begin{proof}
Let $\{(\alpha_i, \rho_i)\}_{i=1}^{k}$, $-1 \leq \alpha_1 < \alpha_2 <\cdots <\alpha_{k} < 1$,
$\rho_i>0$ for $i=1,2,\ldots,k$, be a  $1/N$-quadrature rule for $\Lambda$.
Let $C \subset \mathbb{S}^{n-1}$ be an optimal code for \eqref{snm} (also referred to as {\em best-packing code}), that is a code that minimizes the largest inner product. 
With $h_a(t):=\exp(at)$,  $a>0$, we have the following estimations
\[ N(N-1)e^{as(n,N)}\geq E_{h_a} (C)\geq  \mathcal{E}_{h_a} (n,N) \geq N^2 \sum_{i=1}^{k} \rho_i e^{a\alpha_i}\geq N^2 \rho_k e^{a\alpha_k}. \]
As  the inequality holds for all $a$ and $\rho_k$ is independent of $a$, we conclude that $s(n,N)\geq \alpha_k$.
\end{proof}

We next recall the linear programming upper bound of the quantity $\mathcal{A}(n,s)$ given in \eqref{ANS}.

\begin{theorem}
\label{thm-lp-sph} (LP bound for spherical codes, \cite{DGS,KL})
Let $n \geq 2$, $s \in [-1,1)$ and $g(t)=\sum_{i=0}^r g_i P_i^{(n)}(t) \in B_{n,s}$, where $B_{n,s}$ is defined in \eqref{gform}. Then  $\mathcal{A}(n,s) \leq g(1)/g_0$.
\end{theorem}

We note that Definition \ref{def-ULB-space} may be extended to positive real $N\geq 2$. We use this setting in the maximum cardinality problem \eqref{ANS}.

\subsection{Levenshtein's framework and ULB}

Of particular importance is the case when the subspace in subsection 2.2 is $\mathcal{P}_\tau$.
For this purpose we briefly introduce Levenshtein's framework (see \cite{Lev}, \cite[Chapter 5]{BHS}).
First, we recall two classical notions. The {\em Delsarte-Goethals-Seidel lower
bound} $D(n,\tau)$ on
\[ \mathcal{B}(n,\tau):=\min\{|C| : C \subset \mathbb{S}^{n-1} \mbox{ is a spherical $\tau$-design}\} \]
is given by (cf. \cite{DGS})
\begin{equation}
\label{DGS-bound}
 \mathcal{B}(n,\tau) \geq D(n,\tau) := {n+k-2+\varepsilon \choose n-1}+{n+k-2 \choose n-1},
\end{equation}
for every $\tau=2k-1+\varepsilon$, $\varepsilon \in \{0,1\}$. Hereafter we use the parameter $\varepsilon \in \{0,1\}$ to present simultaneously
the cases of odd $\tau=2k-1$ (where $\varepsilon=0$) and even $\tau=2k$ (where $\varepsilon=1$). Note that
\[ D(n,\tau)= p^\varepsilon \sum_{i=0}^{k-1+\varepsilon} r_i,  \]
where $p=1-(-1)^kP_{k}^{1,0}(-1)$ (see \cite[Theorem 2.16]{Lev}).

The {\em Levenshtein upper bound} on $\mathcal{A}(n,s)$ can be described as follows.
Let $t_i^{a,b}$ be the greatest zero of $P_i^{a,b}(t)$, $t_0^{1,1}:=-1$, and denote by $I_\tau$ the interval
\[ I_\tau:=\left[t_{k-1+\varepsilon}^{1,1-\varepsilon},t_k^{1,\varepsilon}\right]. \]
Applying Theorem \ref{thm-lp-sph} for suitable polynomials (see \eqref{lev-poly} below)
out of the class  $B_{n,s,}$, $s \in I_\tau$, defined in \eqref{gform}, Levenshtein (see \cite{Lev}) proved that \begin{equation}
\label{L-bound}
 \mathcal{A}(n,s) \leq
L_{\tau}(n,s):= \left(1-\frac{P_{k-1+\varepsilon}^{1,0}(s)}{P_k^{0,\varepsilon}(s)}\right) \sum_{i=0}^{k-1+\varepsilon} r_i,  \ \
\forall s \in I_{\tau}.       \end{equation}

A very important connection between the bounds (\ref{DGS-bound}) and (\ref{L-bound}) is given by the equalities
\begin{equation}
\label{L-DGS1}
L_{\tau-1-\varepsilon}(n,t_{k-1-\varepsilon}^{1,1-\varepsilon})=
L_{\tau-\varepsilon}(n,t_{k-1-\varepsilon}^{1,1-\varepsilon}) = D(n,\tau-\varepsilon)
\end{equation}
at the ends of the intervals $I_\tau$.

The strict monotonicity in $\tau$ of $D(n,\tau)$ implies that for every fixed dimension $n$ and cardinality $N$ there is a unique
\begin{equation}
\label{tauNn}
\tau:=\tau(n,N) \mbox{ such that } N \in (D(n,\tau),D(n,\tau+1)].
\end{equation}
Further, with $k= \left\lceil \frac{\tau+1}{2} \right\rceil$, let
$\alpha_{k+\varepsilon}=s$ be the unique solution of
\begin{equation} \label{Lns_eq}
N=L_{\tau}(n,s), \ s \in I_\tau,
\end{equation}
which exists because of the relations \eqref{L-DGS1} and the strict monotonicity in $s$ of $L_{\tau}(n,s)$.
Then, as described by Levenshtein in \cite[Section 5]{Lev}, there exist
uniquely determined quadrature nodes
\begin{equation} \label{quadrature_nodes}
-1 \leq \alpha_1 < \alpha_2 < \cdots <\alpha_{k+\varepsilon} < 1
\end{equation}
and corresponding positive weights
\begin{equation} \label{quadrature_weights}
\rho_1,\rho_2,\ldots,\rho_{k+\varepsilon},
\end{equation}
such that the following $1/N$-quadrature rule (see Definition  \ref{1/N-quadrature}) holds:
\begin{equation}
\label{defin_qf}
f_0= \frac{f(1)}{N}+\sum_{i=1}^{k+\varepsilon} \rho_i f(\alpha_i), \ \ \forall  f\in \mathcal{P}_{\tau}
\end{equation}
(this is the first level $1/N$-quadrature rule, exact for $\mathcal{P}_{\tau}$).
The numbers $\alpha_i$, $i=1,2,\ldots,k+\varepsilon$, are the roots of the equation
\begin{equation}
\label{def-alfi}
(t+1)^{\varepsilon} \left(P_k^{1,\varepsilon}(t)P_{k-1}^{1,\varepsilon}(\alpha_{k+\varepsilon}) -
P_{k-1}^{1,\varepsilon}(t)P_k^{1,\varepsilon}(\alpha_{k+\varepsilon})\right)=0.
\end{equation}
In fact, the nodes $\{\alpha_i\}$
are the roots of the Levenshtein polynomials $f_{\tau}^{(n,\alpha_{k+\varepsilon})}(t)$
(see \cite[Equations (5.81) and (5.82)]{Lev}) used
for obtaining \eqref{L-bound}; i.e.,
\begin{equation}
\label{lev-poly}
f_{\tau}^{(n,\alpha_{k+\varepsilon})}(t)=(t+1)^{\varepsilon}
\left((P_k^{1,\varepsilon}(t)P_{k-1}^{1,\varepsilon}(\alpha_{k+\varepsilon}) -
P_{k-1}^{1,\varepsilon}(t)P_k^{1,\varepsilon}(\alpha_{k+\varepsilon})\right)^2.
\end{equation}

The $1/N$-quadrature
rule given by \eqref{defin_qf} is an important ingredient in the derivation and investigation of the ULB from \cite{BDHSS}
and their test functions \eqref{TestFunctions}.

Levenshtein's bound \eqref{L-bound} is usually utilized to derive the bound in Theorem \ref{LevUpper} (see also Theorem \ref{LevULBCombined} that establishes that $\mathcal{P}_{\tau}$ is a ULB-space). The largest solution $\alpha_{k+\varepsilon}$ of the equation $N=L_\tau(n,s)$, $\tau=2k-1+\varepsilon$, $\varepsilon \in \{0,1\}$,
is at most $s(n,N)$; 
\begin{equation}
\label{snm-L}
s(n,N) \geq \alpha_{k+\varepsilon}
\end{equation}
which we refer to as the first-level Levenshtein bound for $s(n,N)$; Theorem \ref{lev_bound-imrovement} provides a second-level bound for $s(n,N)$.

\subsection{Hermite interpolation for energy bounds}

Following \cite{Y,CK} (see also \cite{CW,BDHSS}) we use Hermite interpolation to construct a polynomial $f$ upper bounded
by the potential $h$.

\begin{definition}
\label{hermite-notation}
Let $T=\{t_1 \leq t_2 \leq \cdots \leq t_\ell\}$ be an interpolation (multi)set in the interval $[-1,1)$.  We say that two functions $g$ and $h$ {\em agree on $T$}
if for each $t\in T$ with multiplicity $m_t$   the $j$-th order derivatives of $g$ and $h$ exist and agree at $t$ for $j=0,1,2\ldots, m_t-1$.  If $g$ and $h$ agree on $T$, we write $g|_T=h|_T$.
For a polynomial $g$ of degree  $\ell$  with  roots $T$ (counted with their multiplicities) in the interval $[-1,1)$ and a sufficiently smooth function $h$, we denote by $H(h;T)=H(h;g)$  the  Hermite interpolating polynomial of degree at most $\ell-1$   to be the unique polynomial of degree at most $\ell-1$  that agrees with $h$ on $T$ (or the roots of $g$).
Further, for a given subspace $\Lambda \subset C([-1,1])$, we denote by $H_\Lambda(h;T)=H_\Lambda(h;g)$ any function in $\Lambda$ that agrees with $h$ on $T$ (or the roots of $g$), if such exists.
\end{definition}

With the above notation, we remark that the ULB result was obtained in \cite{BDHSS} by using the polynomials
\begin{equation}
\label{OptPoly_Lev}
f =H(h; (\cdot-s)f_{\tau}^{(n,s)} ).
\end{equation}

\begin{theorem}[\cite{BDHSS}, Theorem 3.1]\label{thm3.2} Let $n$, $N$
be fixed and $h$ be an absolutely monotone potential on $[-1,1]$.
Suppose that $\tau=\tau(n,N)=2k-1+\varepsilon$ is as in \eqref{tauNn} and let
the associated $1/N$-quadrature nodes and weights  $\alpha_i$ and $\rho_i$, $i=1,2,\ldots,k+\varepsilon$,
be as in \eqref{Lns_eq}--\eqref{def-alfi}.
Then
\begin{equation}
\label{bound_odd}
 {\mathcal E}_h(n,N) \geq R_{\tau}(n,N;h):=N^2\sum_{i=1}^{k+\varepsilon} \rho_i h(\alpha_i) .
 \end{equation}
Moreover, the polynomials $f(t)$ defined by \eqref{OptPoly_Lev} provide the unique
optimal solution of the linear program \eqref{Wdef} for the subspace $\Lambda=\mathcal{P}_\tau$, and consequently
\begin{equation}\label{LP_optimality}
\mathcal{W}_{h,\mathcal{P}_\tau}(n,N)=R_{\tau}(n,N;h).
\end{equation}
\end{theorem}

Using Definition \ref{def-ULB-space} and Theorem \ref{thm3.2} the Levenshtein bounds and our ULB can be unified, in a sense.

\begin{theorem}\label{LevULBCombined}
Let $n\geq 2$ and $\tau$ be positive integers.  If $N\in \left(D(n,\tau),D(n,\tau+1)\right]$, then
 $\mathcal{P}_{\tau}$ is a ULB-space for dimension $n$ and cardinality $N$.
\end{theorem}

\begin{proof}
For every $N \in  \left(D(n,\tau),D(n,\tau+1)\right]$ (even if $N$ is real) we have the Levenshtein $1/N$-quadrature rule \eqref{defin_qf} with $N=L_\tau(n,s)$. Then
the existence and uniqueness of the ULB solution for (ii) of Definition \ref{def-ULB-space} follow from Theorem \ref{thm3.2} (see for details \cite{BDHSS}).
\end{proof}

The optimality property \eqref{LP_optimality} of the polynomials \eqref{OptPoly_Lev} implies that the bound \eqref{bound_odd} can be improved by
linear programming only if the degree of the improving polynomial is at least $\tau+1$. It follows from Theorem \ref{THM_subspace}
that for improvements some negative test functions $Q_j^{(n,N)}$, $j \geq \tau+1$, must arise. Furthermore, since $Q_j^{(n,N)}>0$ for $j \in \{\tau+1,\tau+2\}$
(see Theorem 4.3 in \cite{BDHSS}) and every $N$, any second level ULB for $ {\mathcal E}_h(n,N)$
will require polynomials of degree at least $\tau+3$. The same is true for the second level bounds on $\mathcal{A}(n,s)$

In Section 3 we build a framework for the derivation of a second level ULB for
the case when the improving polynomial has degree $\tau+4$, where at least one test function $Q_{\tau+3}^{(n,N)}$ or $Q_{\tau+4}^{(n,N)}$ is negative.

\subsection{Positive definite Hermite interpolants}
Our goal in this section is to present the interpolation framework needed to verify if a subspace is a ULB-subspace.
Let $T=\{t_1 \leq t_2 \leq \cdots \leq t_\ell\}$ be an interpolation (multi)set and
\begin{equation}
\label{partial-1}
g_j(t)=(t-t_1)\cdots(t-t_j), \ j=1,\ldots,\ell, \ \ g_0(t):=1.
\end{equation}
We are interested in solutions to Hermite interpolation problems on $T$
for certain subspaces of polynomials, which we address in the following two statements. General results can be found in \cite{Hau78,Ike73,MP00}.

\begin{lemma}
\label{lem1}
Let $\Lambda$ be a given subspace of $C([-1,1])$, $T$  a multiset in $[-1,1)$ with finite cardinality $\ell$, and     $g_j$   the partial product defined in \eqref{partial-1} for $j=0,1,\ldots,\ell$.  If for each $j=0,1,\ldots,\ell$ there exists an interpolant $H_\Lambda(g_j;T)$, then
for all sufficiently smooth functions $h$, the function
\begin{equation}
\label{ex-h-partial}
H_{\Lambda,T}(h) :=\sum_{j=0}^{\ell-1} h[t_1,\ldots,t_j] H_\Lambda(g_{j};T),
\end{equation}
where $h[t_1,\ldots,t_j]$ denotes the divided difference in the listed nodes, agrees with $h$ on $T$; i.e., $H_{\Lambda,T}(h)=H_\Lambda(h;T)$.
\end{lemma}

\begin{proof}
Considering the standard Newton interpolating polynomial to $h$ in $\mathcal{P}_{\ell-1}$,
\[ J_{\ell-1}(t):=\sum_{j=0}^{\ell-1} h[t_1,\ldots,t_j]g_j(t), \]
we  note that $J_{\ell-1}(t)|_{T}=h(t)|_{T}$. As $g_{j-1}|_T=H_\Lambda(g_{j-1};g_\ell)|_T$,
linearity of interpolation yields \eqref{ex-h-partial}.
\end{proof}

The next theorem provides   sufficient conditions for the  interpolant defined in \eqref{ex-h-partial} to belong to $A_{n,h}$.

\begin{theorem}
\label{implications-by-pps}
Let $\Lambda$ be a given subspace of $C([-1,1])$, $h$ be absolutely monotone on $[-1,1]$, $T$ a finite multiset in $[-1,1)$ and $g_j$   as in \eqref{partial-1}. If
$H_\Lambda(g_{j};T)$ exist and belong to $A_{n,g_{j}}$, $j=0,1,\ldots,\ell$, then
$H_\Lambda(h;T)$ defined by \eqref{ex-h-partial}  belongs to $A_{n,h}$.
\end{theorem}

\begin{proof}
Since $h(t)$ is absolutely monotone, the divided differences $h[t_1,\ldots,t_j]$ are nonnegative for every $j$ (see, e.g. \cite[Cor. 3.4.2]{D}).
Hence, the interpolant $H_\Lambda(h;T)$ is positive definite since it is in the positive cone of the interpolants
$H_\Lambda(g_{j};T)$, $j=0,1,\ldots,\ell-1$, which are positive definite by assumption.

We now prove that $H_\Lambda(h;T) \leq h(t)$ for every $t \in [-1,1)$. Let, as in Lemma \ref{lem1},
$J_{\ell-1}(t)=H(h;g_\ell)$.
Then $h(t) \geq J_{\ell-1}(t)$ by consecutive applications of Rolle's theorem (see, for example, \cite[Lemma 9]{CW}), or,
alternatively, by the remainder formula for the Hermite interpolation.
Thus it follows from the representation
\[ h(t)-H_\Lambda(h;g_\ell)=(h(t)-J_{\ell-1}(t))+(J_{\ell-1}(t)-H_\Lambda(h;T)) \]
that it is enough to prove that $J_{\ell-1}(t)-H_\Lambda(h;g_\ell) \geq 0$ for every $t \in [-1,1)$.
Since $H_\Lambda(g_{j-1};g_\ell) \in A_{n,g_{j-1}}$ by the assumption, we have
\[ (t-t_1)\cdots(t-t_{j-1}) - H_\Lambda(g_{j-1};g_\ell) \geq 0. \]
Using again that $h[t_1,\ldots,t_j] \geq 0$, the desired result follows from the formula of Lemma \ref{lem1}.
\end{proof}


%
%

In the next section we develop a framework for proving that certain spaces of polynomials
of degrees at most $\tau(n,N)+4$ are also ULB-spaces.

\section{Second level ULB: Lifting the Levenshtein framework}
\label{section-s2a2}

We describe in detail our approach in the case when $n$ and $N$ are such that 
\begin{equation}\label{odd_case}
\tau(n,N)=2k-1,\quad  \mbox{and} \quad Q_{2k+2}^{(n,N)}<0\ \mbox{or}\  Q_{2k+3}^{(n,N)}<0.
\end{equation}
In this case  the {\em skip-two/add-two} subspace defined by \eqref{Lambdakn} becomes
\begin{equation}
\label{def-L-odd}
\Lambda_{n,k} := \mathcal{P}_{2k-1}
\oplus  \mbox{span} \left(P_{2k+2}^{(n)}, P_{2k+3}^{(n)}\right).
\end{equation}
The goal will be to derive conditions on $n$ and $N$ for $\Lambda_{n,k}$ to be a ULB-space (see Definition \ref{def-ULB-space}). We focus separately on the existence of a $1/N$-quadrature and on an admissible polynomial interpolant in the subspace associated with this quadrature.

\subsection{Existence of a $1/N$-quadrature rule for the skip-two/add-two subspace $\Lambda_{n,k}$.}

First, we focus on necessary conditions for the existence of a $1/N$-quadrature rule exact on  $\Lambda_{n,k}$ under the assumption \eqref{odd_case}. 
\begin{lemma}
Any $1/N$-quadrature rule that is exact on $\Lambda_{n,k}$ has at least $k+1$ distinct nodes.
\end{lemma}

\begin{proof} Let $\{ \beta_i,\theta_i\}_{i=1}^\ell$ be such a quadrature rule. If $\ell < k$, then $(t-\beta_1 )^2 \dots (t-\beta_\ell )^2(1-t) \in \Lambda_{n,k}$ and hence
\[
\int_{-1}^1 (t-\beta_1 )^2 \dots (t-\beta_\ell )^2(1-t) \, d\mu(t) =0,
\]
which is absurd. If $\ell=k$, then we are exactly in the condition of the Levenshtein quadrature with $N=L(n,\beta_\ell)$. The remaining nodes $\{\beta_i\}_{i=1}^{k-1}$ are uniquely determined and satisfy \eqref{def-alfi}, and therefore, the quadrature is the Levenshtein quadrature.
However, if the Levenshtein quadrature is exact for $P_{2k+2}^{(n)}$ and $ P_{2k+3}^{(n)}$, then the test functions $Q_{2k+2}^{(n,N)}$ and $Q_{2k+3}^{(n,N)}$ vanish, which contradicts the hypothesis.
\end{proof}

Suppose now that $\{ (\beta_i , \theta_i )\}_{i=1}^{k+1}$ is a $1/N$-quadrature rule with exactly $k+1$ nodes that is exact on the subspace $\Lambda_{n,k}$, namely, $-1\leq\beta_1<\dots<\beta_{k+1}<1$, $\theta_i>0, \ i=1,\dots,k+1$, and
\begin{equation}
\label{quad-2-level}
f_0=\int_{-1}^1f(t)d\mu(t)=\frac{f(1)}{N}+\sum_{i=1}^{k+1} \theta_i f(\beta_i)
\end{equation}
holds for every polynomial from the subspace $\Lambda_{n,k}$ defined in \eqref{def-L-odd}. We define the {\sl partial products} associated with the $1/N$-quadrature rule as 
\begin{equation}\label{ParProd}
q_j(t):=(t-\beta_1)\cdots(t-\beta_j), \quad j=1,2,\ldots,k+1, \quad q_0(t)\equiv 1.
\end{equation}
The following theorem provides some insight on how to determine this quadrature and its relation to the Levenshtein $1/N$-quadrature $\{ (\alpha_i,\rho_i )\}_{i=1}^k$.

\begin{theorem}\label{IL}
If $\{ (\beta_i , \theta_i )\}_{i=1}^{k+1}$ is a $1/N$-quadrature rule exact on the subspace $\Lambda_{n,k}$, then there are constants $c_1$,  $c_2$, and $c_3$ such that (compare to \eqref{def-alfi} with $\varepsilon=0$)
 \begin{equation}
\label{beti-1}
a_{k+1,k+1}^{1,0} q_{k+1} (t)=P_{k+1}^{1,0}(t)+c_1P_{k}^{1,0}(t)+c_2P_{k-1}^{1,0}(t)+c_3P_{k-2}^{1,0}(t),
\end{equation}
where $a_{k+1,k+1}^{1,0}$ denotes the leading coefficient of the polynomial $P_{k+1}^{1,0}(t)$. 
Moreover, $(\beta_i)_{i=1}^{k+1}$ interlace with the nodes $(\alpha_i)_{i=1}^k$ of the Levenshtein quadrature for $\mathcal{P}_{2k-1}$ given in \eqref{quadrature_nodes}-\eqref{defin_qf}; i.e.,
\begin{equation} \label{interlacing}
-1\leq \beta_1<\alpha_1<\dots<\beta_k<\alpha_k<\beta_{k+1}<1 .
\end{equation}
\end{theorem}

\begin{proof} Since the degrees
of the polynomials $q_{k+1}(t)(1-t)P_i^{1,0}(t)$ for $i=0,1,\ldots,k-3$ do not exceed $2k-1$, these polynomials belong to $\Lambda_{n,k}$.
As they annihilate the quadrature, we obtain that $q_{k+1}(t)$ is orthogonal to the adjacent polynomials $P_i^{1,0}(t)$, $i=0,1,\ldots,k-3$,
with respect to the adjacent measure $d\nu^{1,0}(t):=(1-t) d \mu(t)$ (see \eqref{nuab}). Therefore, the expansion of $q_{k+1}(t)$
in terms of the polynomials $P_i^{1,0}$ has at most four non-zero terms and \eqref{beti-1} follows.

To prove the interlacing property \eqref{interlacing} of the nodes, we use suitable polynomials of degree $2k-1$ simultaneously in the
quadratures \eqref{defin_qf} and \eqref{quad-2-level}. We first prove that $\beta_{k+1}>\alpha_{k}$ and $\beta_1<\alpha_1$.
Applying \eqref{defin_qf} and \eqref{quad-2-level} to the Levenshtein polynomial 
\[f_{2k-1}^{(n,\alpha_{k})}(t)=
(t-\alpha_1)^2(t-\alpha_2)^2\cdots(t-\alpha_{k-1})^2(t-\alpha_k)\]  
yields that
\[ \sum_{i=1}^{k+1} \theta_i f_{2k-1}^{(n,\alpha_{k})}(\beta_i)=0. \]
If $\beta_{k+1} \leq \alpha_{k}$, then the last sum consists of $k+1$ nonpositive terms and they cannot be all equal to zero
(otherwise $\{\beta_1,\beta_2,\ldots,\beta_{k+1}\} \subseteq \{\alpha_1,\alpha_2,\ldots,\alpha_k\}$, which is impossible). Thus,
the sum is negative, a contradiction that yields $\beta_{k+1}>\alpha_k$.

Similarly, \eqref{defin_qf} and \eqref{quad-2-level} applied to the polynomial $f_1(t):=(t-\alpha_k)f_{2k-1}^{(n,\alpha_{k})}(t)/(t-\alpha_1)$ shows that
\[ \sum_{i=1}^{k+1} \theta_i f_1(\beta_{i})=0, \]
implying a contradiction if $\beta_{1} \geq \alpha_{1}$; i.e., we have $\beta_1<\alpha_1$.

We proceed with separating $\beta_2,\ldots,\beta_k$ to interlace $\alpha_1,\ldots,\alpha_k$ as follows. Consider the polynomial
\[ f_2(t):=\frac{(t-\beta_1)(t-\alpha_k)f_{2k-1}^{(n,\alpha_{k})}(t)}{(t-\alpha_1)(t-\alpha_2)}. \]
Applying both \eqref{defin_qf} and \eqref{quad-2-level} to $f_2(t)$ we conclude as above that
there is some $\beta_i$, $i \in \{2,3,\ldots,k\}$, in the interval $(\alpha_1,\alpha_2)$. Continuing
this way (consecutively getting $\beta_i$'s in intervals $(\alpha_j,\alpha_{j+1})$)
we conclude that every interval  $(\alpha_i,\alpha_{i+1})$, $i=1,2,\ldots,k-1$, contains
some $\beta_j$. Since the number of $\beta_j$'s to be placed that way is equal to the number of the
intervals for placing them, the proof of the interlacing is completed.
\end{proof}

For arbitrary real numbers $c_1$, $c_2$, $c_3$, set 
\begin{equation} \label{rk+1}r_{k+1}(t):=P_{k+1}^{1,0}(t)+c_1P_{k}^{1,0}(t)+c_2P_{k-1}^{1,0}(t)+c_3P_{k-2}^{1,0}(t).\end{equation}
We seek necessary conditions on the coefficients $c_1$, $c_2$, and $c_3$ in \eqref{beti-1}, so that a $1/N$-quadrature \eqref{quad-2-level} exists. 
Consider polynomials of degree $2k+3$ that lie in the subspace $\Lambda_{n,k}$ and have the form
\begin{equation} \label{Poly}
L(t):=r_{k+1}(t) (1-t) \left[ d_0P_{k+1}^{1,0}(t)+d_1P_{k}^{1,0}(t)+d_2P_{k-1}^{1,0}(t)+d_3P_{k-2}^{1,0}(t)\right] \in \Lambda_{n,k}, 
\end{equation}
where $d_i$, $i=0,1,2,3$, are chosen so that the required inclusion holds.
Since $\Lambda_{n,k} \perp \mbox{span} (P_{2k}^{(n)}, P_{2k+1}^{(n)})$, we have
\begin{equation}
\label{fi=0_1}
\langle L, P_{2k}^{(n)} \rangle = \langle L, P_{2k+1}^{(n)} \rangle =0.
\end{equation}
We use the two equations from \eqref{fi=0_1} to express the coefficients $d_2$ and $d_3$ as functions of $d_0$ and $d_1$; i.e., we
obtain $d_2=d_2(d_0,d_1)$ and $d_3=d_3(d_0,d_1)$. Since this procedure is applicable to all polynomials $L(t)$
in \eqref{Poly}, we can substitute $(d_0,d_1)=(1,0)$ and $(0,1)$ to find
\begin{equation}
\label{d23}
d_2^{(1)}:=d_2(1,0), \ d_3^{(1)}:=d_3(1,0), \ d_2^{(2)}:=d_2(0,1), \ d_3^{(2)}:=d_3(0,1).
\end{equation}
Note that \eqref{d23} gives $d_i^{(j)}$ as functions of $c_1$, $c_2$, and $c_3$.

Should the roots of $r_{k+1}(t)$ be all simple and lie in $[-1,1)$, we can apply the quadrature \eqref{quad-2-level} to $L(t)$ with these $d_i^{(j)}$ to obtain
the equations
\begin{equation}
\label{duad-forL-1}
\| P_{k+1}^{1,0} \|_{1,0}^2 + c_2d_2^{(1)}\|P_{k-1}^{1,0} \|_{1,0}^2+c_3d_3^{(1)} \| P_{k-2}^{1,0} \|_{1,0}^2=0,
\end{equation}
and 
\begin{equation}
\label{duad-forL-2}
c_1\| P_k^{1,0} \|_{1,0}^2 + c_2d_2^{(2)}\|P_{k-1}^{1,0} \|_{1,0}^2+c_3d_3^{(2)} \| P_{k-2}^{1,0} \|_{1,0}^2=0.
\end{equation}

Further, we apply the $1/N$-quadrature \eqref{quad-2-level} to the polynomial $r_{k+1}(t)$ to get the linear equation
\begin{equation}
\label{duad-forq}
I_{k+1}+c_1I_k+c_2I_{k-1}+c_3I_{k-2}=\frac{1+c_1+c_2+c_3}{N},
\end{equation}
where
\[ I_j := \int_{-1}^1 P_j^{1,0} (t) d \mu (t) =  \left(\sum_{i=0}^j r_i\right)^{-1} \]
from the Christoffel-Darboux formula.

Thus, we obtain the equations \eqref{duad-forL-1}-\eqref{duad-forq}
for the coefficients $c_1$, $c_2$, and $c_3$. We first express $c_1$ as a
linear function of $c_2$ and $c_3$ from \eqref{duad-forq}.
Then $c_2$ is derived uniquely as a function of $c_3$. The final polynomial equation for $c_3$ has degree $6$. This also allows us to conclude via interval analysis the existence of exact solutions close to numerical ones. 
There are multiple choices for $c_3$ and we check all of them to obtain a polynomial $r_{k+1}(t)$. In general, the roots of $r_{k+1}(t)$ need not be all real, there are some cases when there is a complex conjugate pair (as the lemma below shows, there can be no more than one such pair). If the roots are all real and simple, and belong to $[-1,1)$, then these may serve as $1/N$-quadrature nodes  $\beta_1,\beta_2,\ldots,\beta_{k+1}$. 

\medskip

Given the nodes $\{ \beta_i\}_{i=1}^{k+1}$, the corresponding weights $\theta_1,\theta_2,\ldots,\theta_{k+1}$ can be computed
 using in \eqref{quad-2-level} the Lagrange basis polynomials
\[ \ell_i(t)=(t-1)\prod_{j \neq i} (t-\beta_j)=\frac{(t-1)q_{k+1}(t)}{a_{k+1,k+1}^{1,0}(t-\beta_i)}, \]
as follows
\begin{equation}
\label{duad-forq2}
\theta_i = \frac{1}{\ell_i(\beta_i)} \int_{-1}^1 \ell_i(t) d \mu(t), \ \ i=1,2,\ldots,k+1.
\end{equation}
If the weights are all positive, then the roots of $r_{k+1}(t)$ serve indeed as $1/N$-quadrature nodes. We summarize the discussion in the following theorem.

\begin{theorem}\label{QRthm}
Suppose $c_1$, $c_2$, and $c_3$ are real solutions to the system \eqref{duad-forL-1}-\eqref{duad-forq}, where $d_2^{(i)}, d_3^{(i)}$, $i=1,2$ are found utilizing \eqref{fi=0_1} and \eqref{d23}. Then  $r_{k+1}(t)$ defined in \eqref{rk+1} has at least $k-1$ distinct real roots, of which at least $k-2$ are in the interval $[-1,1)$. If all of its roots are real and simple, belong to $[-1,1)$, and if the associated weights \eqref{duad-forq2} are all positive, then the collection $\{(\beta_i,\theta_i)\}_{i=1}^{k+1}$ forms a $1/N$-quadrature rule for $\Lambda_{n,k}$.
\end{theorem}

\begin{proof}
We first show that $r_{k+1}(t)$ has at least $k-2$ sign changes in $[-1,1)$. Indeed, if it has less, then there is a polynomial $p_{k-3}$ of degree at most $k-3$ that has the same sign changes. But then $r_{k+1}(t)p_{k-3}(t)(1-t)$ does not change sign in $[-1,1]$, and hence
\[\int_{-1}^1 r_{k+1}(t)p_{k-3}(t) (1-t)\, d\mu(t) \not=0. \]
This is a contradiction since $r_{k+1}(t)$ is orthogonal to $p_{k-3} (t)$ with respect to $ d\nu^{1,0}(t)$ (see \eqref{nuab} and \eqref{rk+1}). Hence, $r_{k+1}(t)$ has at most one pair of complex conjugate roots. As the total number of roots, counting multiplicity, is exactly $k+1$, and as we already have at least $k-2$ sign changes in $[-1,1]$, we conclude that $r_{k+1}$ has at least $k-1$ distinct real roots. 

Suppose now that all roots of $r_{k+1}(t)$ are real and simple and lie in $[-1,1)$, and that the weights \eqref{duad-forq2} are all positive. Let $f(t)\in \Lambda_{n,k}$ be arbitrary. We want to show that the $1/N$-quadrature \eqref{quad-2-level} holds. We first observe, that by the choice of the weights $\theta_i$ from \eqref{duad-forq2} and the condition \eqref{duad-forq} the $1/N$-quadrature rule \eqref{quad-2-level} holds for the Lagrange basis polynomials of degree $k+1$ associated with the nodes $\{\beta_1,\dots, \beta_{k+1},1\}$, and hence it holds for all polynomials in $\mathcal{P}_{k+1}$. We expand $f(t)$
\begin{equation}\label{f_exp}
f(t)=(1-t)r_{k+1}(t)u_{k+1}(t)+v_{k+1}(t) 
\end{equation}
for some $u_{k+1},v_{k+1}\in \mathcal{P}_{k+1}$. Consider the two polynomials
\[ y_1(t):=P_{k+1}^{1,0}(t)+d_2^{(1)}P_{k-1}^{1,0}(t)+d_3^{(1)}P_{k-2}^{1,0}(t), \quad y_2(t):=P_{k}^{1,0}(t)+d_2^{(2)}P_{k-1}^{1,0}(t)+d_3^{(2)}P_{k-2}^{1,0}(t).\]
Observe that by the choice of $d_i^{(j)}$ and \eqref{fi=0_1} we have 
\begin{equation}\label{yortho}
\langle (1-t)r_{k+1}(t)y_j(t),P_{2k}(t)\rangle=0, \quad \langle (1-t)r_{k+1}(t)y_j(t),P_{2k+1}(t)\rangle=0,\quad j=1,2,
\end{equation} 
or  $(1-t)r_{k+1}(t)y_j(t)  \in \Lambda_{n,k}$ for $j=1,2$. We can express $u_{k+1}(t)$ as 
\[ u_{k+1}(t)=Ay_1(t)+By_2(t)+CP_{k-1}^{1,0}(t)+DP_{k-2}^{1,0}(t)+w_{k-3}(t), \]
for some $w_{k-3}\in \mathcal{P}_{k-3}$. From  $\langle f,P_{2k+1}\rangle=0$ and \eqref{yortho} we have that 
\[0=\langle r_{k+1}(t) (1-t)(CP_{k-1}^{1,0}(t)+DP_{k-2}^{1,0}(t)+w_{k-3}(t)),P_{2k+1}(t)\rangle=C\int_{-1}^1 P_{k+1}^{1,0}(t) P_{k-1}^{1,0}(t) P_{2k+1}(t) (1-t) \, d\mu(t).\]
Since $P_{k+1}^{1,0}(t) P_{k-1}^{1,0}(t) (1-t)$ is of exact degree $2k+1$ and has nonzero leading coefficient, the integral is non-zero and we conclude that $C=0$.
Similarly $\langle f,P_{2k}\rangle =0$ and \eqref{yortho} imply that $D=0$. Equations 
\eqref{duad-forL-1} and \eqref{duad-forL-2} now yield that
\[\int_{-1}^1 (1-t)r_{k+1}(t)u_{k+1}(t) \, d\mu(t)=0.\]
Utilizing this equation, the fact that the first term in the sum on the right-hand side of \eqref{f_exp} annihilates the quadrature sum, and that the quadrature holds for $v_{k+1}$, we conclude \eqref{quad-2-level}, which completes the proof.
\end{proof}

\subsection{Existence of Hermite interpolant to $h(t)$ in the skip-two/add-two subspace $\Lambda_{n,k}$.}

We shall use Lemma \ref{lem1} and Theorem \ref{implications-by-pps} to determine sufficient conditions for the existence of a $\Lambda_{n,k}$-LP-extremal polynomial
\begin{equation}
\label{PP-extr}
f^{h}(t)=f_{2k+3}^hP_{2k+3}^{(n)}(t)+f_{2k+2}^hP_{2k+2}^{(n)}(t)+\sum_{i=0}^{2k-1} f_i^hP_i^{(n)}(t) \in \Lambda_{n,k} \cap A_{n,h}
\end{equation}
that interpolates the potential function $h(t)$ at the nodes $\{\beta_i\}_{i=1}^{k+1}$.

\begin{lemma}\label{ExistIP}
Suppose $h$ is an absolutely monotone function on $[-1,1)$ and $T$ is a multiset on $(-1,1)$
\[ T=\{\beta_1,\beta_1,\beta_2,\beta_2,\ldots,\beta_{k+1},\beta_{k+1}\} \]
with $q_{k+1}(t)=(t-\beta_1)\dots (t-\beta_{k+1})$ (see \eqref{ParProd}). If
\begin{equation}
\label{existence-cond}
\langle tq_{k+1}^2(t),P_{2k}^{(n)}(t)\rangle \cdot \langle q_{k+1}^2(t),P_{2k+1}^{(n)}(t)\rangle \neq
\langle tq_{k+1}^2(t),P_{2k+1}^{(n)}(t)\rangle \cdot \langle q_{k+1}^2,P_{2k}^{(n)}(t)\rangle ,
\end{equation}
then there exists a unique polynomial $f^h \in \Lambda_{n,k}$ denoted 
\begin{equation}
\label{def-fhd}
f^{h}(t):=H_{\Lambda_{n,k}}(h;q_{k+1}^2)
\end{equation}
that interpolates $h(t)$ at the nodes of the multiset $T$.
\end{lemma}
\begin{proof}
We first prove the uniqueness. If $F_1^h(t)$ and $ F_2^h (t)$ are two such interpolants, then the nodes of the multiset are zeros of the difference, and therefore
\begin{equation}\label{MultisetIntEq}
F_1^h (t)-F_2^h (t)=q_{k+1}^2(t)(At+B).
\end{equation}
Since each $F_i^h (t)$, $i=1,2$ is orthogonal to $P_{2k}^{(n)}(t)$ and $P_{2k+1}^{(n)}(t)$, the constants $A$ and $B$ satisfy the linear system
\begin{eqnarray}
\begin{split}
\label{LinSys}
A \langle tq_{k+1}^2(t),P_{2k}^{(n)}(t)\rangle + B \langle q_{k+1}^2(t),P_{2k}^{(n)}(t)\rangle
&=&0, \\
A \langle tq_{k+1}^2(t),P_{2k+1}^{(n)}(t)\rangle + B \langle q_{k+1}^2(t),P_{2k+1}^{(n)}(t)\rangle
&=&0.
\end{split}
\end{eqnarray}
The condition \eqref{existence-cond} now yields that \eqref{LinSys} has only the trivial solution, which implies the uniqueness.
\smallskip

To prove the existence, as $H_{\Lambda_{n,k} }(g_j;g_{2k+2})=g_j(t)$ for all $j\leq 2k-1$, by the definition \eqref{def-L-odd} of the subspace $\Lambda_{n,k}$, it is enough to establish the existence of
the interpolants $H_{\Lambda_{n,k} }(g_j;g_{2k+2})$ for $j=2k$ and $2k+1$ and apply Lemma \ref{lem1} (note that $g_{2k}(t)=q_k^2(t)$, $g_{2k+1}(t)=q_k(t)q_{k+1}(t)$,
and $g_{2k+2}(t)=q_{k+1}^2(t)$). Similar to \eqref{MultisetIntEq} we find
\begin{eqnarray}
\label{ab-solution1}
\begin{split}
H_{\Lambda_{n,k} }(q_k^2;q_{k+1}^2)-q_{k}^2(t) &=& q_{k+1}^2(t)(A_1t+B_1), \\
H_{\Lambda_{n,k} }(q_kq_{k+1};q_{k+1}^2)-q_k(t)q_{k+1}(t) &=& q_{k+1}^2(t)(A_2t+B_2),
\end{split}
\end{eqnarray}
where the parameters $A_i,B_i$ can be determined by the orthogonality conditions \eqref{fi=0_1}. As in \eqref{LinSys} we get that $A_1$ and $B_1$ satisfy the linear system 
\begin{eqnarray*}
\label{ab-solution2}
A_1 \langle tq_{k+1}^2(t),P_{2k}^{(n)}(t)\rangle + B_1 \langle q_{k+1}^2(t),P_{2k}^{(n)}(t)\rangle
&=& - \langle q_{k}^2(t),P_{2k}^{(n)}(t)\rangle, \\
A_1 \langle tq_{k+1}^2(t),P_{2k+1}^{(n)}(t)\rangle + B_1 \langle q_{k+1}^2(t),P_{2k+1}^{(n)}(t)\rangle
&=& - \langle q_{k}^2(t),P_{2k+1}^{(n)}(t)\rangle.
\end{eqnarray*}
The constants $A_2$ and $B_2$
satisfy similar system, the only difference being the right-hand side. By \eqref{existence-cond} this system has non-zero determinant, which implies the existence (and uniqueness) of the constants $A_i, B_i$, $i=1,2$.
\end{proof}

The proof, and in particular, the equation \eqref{ab-solution1} imply the following corollary.

\begin{corollary} \label{Ineq} In the context of Lemma \ref{ExistIP}, suppose that 
\begin{equation}
\label{ineq-check-a1b1}
\max\{A_i+B_i,B_i-A_i\} \leq 0, \quad i=1,2,
\end{equation}
where $A_i, B_i$ are defined in \eqref{ab-solution1}. Then 
\[ H_{\Lambda_{n,k}}(q_k^2;q_{k+1}^2) \leq q_k^2(t), \ \ H_{\Lambda_{n,k}}(q_kq_{k+1};q_{k+1}^2) \leq q_k(t)q_{k+1}(t). \]
Subsequently, if $h(t)$ is absolutely monotone potential, then $f^h(t)=H_{\Lambda_{n,k}} (h;q_{k+1}^2) \leq h(t)$.
\end{corollary}

\begin{proof}
Since $A_i t+B_i$ in \eqref{ab-solution1} are linear functions, both will be non-positive on $[-1,1]$ if and only if  \eqref{ineq-check-a1b1} holds. The conclusion $H_{\Lambda_{n,k}} (h;q_{k+1}^2) \leq h(t)$ for absolutely monotone potentials now follows from \eqref{ex-h-partial}.
\end{proof}

%
%
%
%
%

To apply Theorem \ref{implications-by-pps} for $f^h$ defined in \eqref{def-fhd} we need to show the positive definiteness of the interpolants
$H_{\Lambda_{n,k}}(g_{j-1};g_{2k+2})$ for $j=1,\ldots,2k+2$. For $j \in \{2k+1,2k+2\}$, 
the special polynomial $H_{\Lambda_{n,k}}(q_k^2;q_{k+1}^2)$ and $H_{\Lambda_{n,k}}(q_kq_{k+1};q_{k+1}^2)$
can be written using \eqref{ab-solution1} as
\[
H_{\Lambda_{n,k}}(q_k^2;q_{k+1}^2)=q_{k}^2(t)+q_{k+1}^2(t)(A_1t+B_1)=q_k^2(t)(1+(t-\beta_{k+1})^2(A_1t+B_1)), \]
\[H_{\Lambda_{n,k}}(q_kq_{k+1};q_{k+1}^2)=q_k(t)q_{k+1}(t)+q_{k+1}^2(t)(A_2t+B_2)=q_k(t)q_{k+1}(t)(1+(t-\beta_{k+1})(A_2t+B_2)).
\]
Then we can verify the positive definiteness of $H_{\Lambda_{n,k}}(q_k^2;q_{k+1}^2)$ and $H_{\Lambda_{n,k}}(q_kq_{k+1};q_{k+1}^2)$
directly  (in particular, $A_1>0$ and $A_2>0$ are necessary conditions) which along with Corollary \ref{Ineq} gives
\[ H_{\Lambda_{n,k}}(q_k^2;q_{k+1}^2) \in A_{n,q_k^2}, \quad H_{\Lambda_{n,k}}(q_k q_{k+1};q_{k+1}^2) \in A_{n,q_k q_{k+1}}. \]

For $j=1,\dots,2k$ we have that $H_{\Lambda_{n,k}}(g_{j-1};g_{2k+2})=g_{j-1} (t)$. Therefore, it is enough to focus on deriving sufficient conditions that guarantee the $(1,0)$-positive definiteness of the partial products $q_i$, $i \leq k$;
i.e., to have nonnegative coefficients in their expansion in terms of the adjacent polynomials $P_i^{1,0}(t)$. Recall that $(1,0)$-positive definiteness implies positive definiteness. Since $g_{2j}(t)=q_j^2(t), g_{2j+1}(t)=q_j (t) q_{j+1}(t)$, $j=0,1,\dots, k-1$ (see \eqref{partial-1} and \eqref{ParProd}), the Krein condition \eqref{krein} will imply that the hypothesis conditions in Theorem \ref{implications-by-pps} hold. 

\medskip

The $(1,0)$-positive definiteness of $q_{k+1}(t)$ is equivalent to the non-negativity of the constants $c_1, c_2, c_3$ in Theorem \ref{QRthm}. The next two lemmas provide sufficient conditions for the $(1,0)$-positive definiteness of $q_k(t)$ and $q_{k-1}(t)$.

\begin{lemma}
\label{thmj=0} In the context of Lemma \ref{ExistIP}, let $d_0:=\theta_{k+1}(1-\beta_{k+1})q_k(\beta_{k+1})$ (note that $d_0>0$). If 
\begin{equation}
\label{c3cond2}
c_3>-\frac{d_0 a_{k+1,k+1}^{1,0} a_{k-2,k-2}^{1,0}r_{k-2}^{1,0}P_{k-1}^{1,0}(\beta_{k+1})}{a_{k-1,k-1}^{1,0}},
\end{equation}
then the polynomial $q_k(t)$ is $(1,0)$-positive. In particular, since the right-hand side of \eqref{c3cond2} is negative, non-negativity of $c_3$ implies
that \eqref{c3cond2} is satisfied.
\end{lemma}

\begin{proof} Let $q_k(t)= \sum_{i=0}^k d_i P_i^{1,0}(t)$.
First, note that $d_k=1/a_{k,k}^{1,0}>0$. For any $\ell \leq k-2$ the degree of the polynomial $q_k(t)P_\ell^{1,0}(t)(1-t)$ is $k+\ell+1 \leq 2k-1$ and this polynomial
belongs to $\Lambda_{n,k}$. Thus, we apply \eqref{quad-2-level} to get
\[ d_\ell \|P_\ell^{1,0} (t)\|^2 =\int_{-1}^1 q_k(t)P_\ell^{1,0}(t)(1-t) d \mu(t) = \theta_{k+1}q_k(\beta_{k+1})P_\ell^{1,0}(\beta_{k+1})(1-\beta_{k+1}) =d_0 P_\ell^{1,0}(\beta_{k+1}) \]
(all other terms are equal to 0; note $d_0$ as in the condition). Since all factors are positive (observe that $\beta_{k+1}>\alpha_k>t_k^{1,0}>t_\ell^{1,0}$)
we have $d_\ell>0$.

Finally, for the last remaining $d_{k-1}$ we consider (as in \cite{CK})
\[ I:=\int_{-1}^1 \frac{q_{k+1}(t)\left(P_{k-1}^{1,0}(t)-P_{k-1}^{1,0}(\beta_{k+1})\right)}{t-\beta_{k+1}} d\nu^{1,0}(t). \]
Comparing coefficients, we obtain
\[  \frac{P_{k-1}^{1,0}(t)-P_{k-1}^{1,0}(\beta_{k+1})}{t-\beta_{k+1}}=\frac{a_{k-1,k-1}^{1,0}}{a_{k-2,k-2}^{1,0}} P_{k-2}^{1,0}(t)+\cdots, \]
whence $I=c_3a_{k-1,k-1}^{1,0}\|P_{k-2}^{1,0}\|_{1,0}^2/(a_{k+1,k+1}^{1,0} a_{k-2,k-2}^{1,0})$.
On the other hand,
\begin{eqnarray*}
&& I=\int_{-1}^1 q_k(t)\left(P_{k-1}^{1,0}(t)-P_{k-1}^{1,0}(\beta_{k+1})\right) d\nu^{1,0}(t), \\
&\iff& \int_{-1}^1 q_k(t)P_{k-1}^{1,0}(t)d\nu^{1,0}(t) =I+ P_{k-1}^{1,0}(\beta_{k+1}) \int_{-1}^1 q_k(t) d\nu^{1,0}(t) \\
&\iff& d_{k-1} \| P_{k-1}^{1,0}\|_{1,0}^2=I+P_{k-1}^{1,0}(\beta_{k+1})d_0>0
\end{eqnarray*}
and we conclude that $d_{k-1}>0$.
\end{proof}

\begin{lemma}
\label{thmj=1}
In the context of Lemma \ref{ExistIP}, if 
\begin{equation}
\label{pd-qk-1}
P_{j}^{1,0}(\beta_k) > - \frac{\theta_{k+1} q_{k-1}(\beta_{k+1})(1-\beta_{k+1})P_{j}^{1,0}(\beta_{k+1})}
{ \theta_{k} q_{k-1}(\beta_k) (1-\beta_k) }
\end{equation}
for every $j \leq k-2$, then the polynomial $q_{k-1}(t)$ is $(1,0)$-positive definite.
\end{lemma}

\begin{proof} Let $q_{k-1}(t)=\sum_{i=0}^{k-1} e_i P_i^{1,0}(t)$.
It is clear that $e_{k-1}>0$. For $j \leq k-2$ we have
\begin{eqnarray*}
e_j &=& \int_{-1}^1 q_{k-1}(t)P_j^{1,0}(t) d \nu^{1,0}(t) \\
    &=& \int_{-1}^1 q_{k-1}(t)P_j^{1,0}(t)(1-t) d \mu(t) \\
    &=& \theta_{k} q_{k-1}(\beta_k)P_j^{1,0}(\beta_k)(1-\beta_k)+\theta_{k+1} q_{k-1}(\beta_{k+1})P_j^{1,0}(\beta_{k+1})(1-\beta_{k+1})>0
\end{eqnarray*}
using \eqref{quad-2-level}.
\end{proof}

\begin{remark}
For small $N$ the requirement \eqref{pd-qk-1} can be replaced with the weaker $\beta_k>t_{j}^{1,0}$.
For example, in three dimensions the last is satisfied for each $j \leq k-2$ and $N \leq 14$, for each $j \leq k-3$ and $N \leq 44$, etc.
\end{remark}

To analyze the remaining partial products $q_i(t)$, $i \leq k-2$, we adapt the approach from \cite[Section 3]{CK} utilizing the $1/N$-quadrature rule \eqref{quad-2-level}. We consider the signed measure $\mu_j(t)$ defined by
\begin{eqnarray*} d \mu_j(t)&=& (\beta_{k+1}-t)(\beta_k-t)\ldots(\beta_{k-j+2}-t)(1-t) d\mu(t)
\end{eqnarray*}
(of course, $\mu_0=\mu$; the cases $j=0$ and $j=1$ were considered above).

\begin{definition}
A signed Borel measure $\eta$ on $\mathbb{R}$ for which all polynomials are integrable
is called {\em positive definite up to degree $ {m}$} if for all real polynomials $p \not\equiv 0$
of degree at most $m$ we have $\int p(t)^2 d \eta(t) > 0$.
\end{definition}

\begin{lemma}
\label{pd-measure}
For $2 \leq j \leq k$, the signed measure $\mu_j(t)$ is positive definite up to degree $k-j$.
\end{lemma}

\begin{proof}
If $f(t)$ is arbitrary polynomial of degree at most $k-1-j/2 \geq k-j$ for $j \geq 2$, then
\begin{eqnarray*}
 \int_{-1}^1 f^2(t) d \mu_j(t) &=& \int_{-1}^1 f^2(t)(\beta_{k+1}-t)(\beta_k-t)\ldots(\beta_{k-j+2}-t)(1-t) d\mu(t) \\
 &=& \sum_{i=1}^{k-j+1} \theta_i f^2(\beta_i)(\beta_{k+1}-\beta_i)(\beta_k-\beta_i)\ldots(\beta_{k-j+2}-\beta_i)(1-\beta_i) \geq 0,
 \end{eqnarray*}
where we used that $f^2(t)(\beta_{k+1}-t)(\beta_k-t)\ldots(\beta_{k-j+2}-t)(1-t) \in \Lambda_{n,k}$ and therefore the quadrature \eqref{quad-2-level} can be applied. The equality can be attained if and
only if $f(\beta_i)=0$ for $i=1,2,\ldots,k-j+1$, which means that $f(t) \equiv 0$ when $\deg(f) \leq k-j$.
This completes the proof for $j \geq 2$.
\end{proof}

\begin{lemma} {\rm (}\cite[Lemma 3.5]{CK}{\rm )}
\label{ortho-poly-sm}
Let the measure $\eta(t)$ be positive definite up to degree $M$. Then there are unique
monic polynomials $p_0, p_1,\ldots, p_{M+1}$ such that $\deg(p_i) = i$ for each $i$ and
\[ \int p_i(t)p_j(t) d \eta(t) = 0 \]
for $i \neq j$. For each $i$, $p_i$ has $i$ distinct real roots, and the roots of $p_i$ and $p_{i-1}$
are interlaced.
\end{lemma}

\begin{proof} The proof adopts standard Gramm-Schmidt orthogonalization and can be found in \cite{CK} or \cite{Si}.
\end{proof}

Combining Lemmas \ref{pd-measure} and \ref{ortho-poly-sm}, we denote by
\[ q_{j,0}(t),q_{j,1}(t),\ldots,q_{j,k-j+1}(t) \]
the unique monic polynomials that are orthogonal with respect to $\mu_j(t)$ and enjoy the properties
\begin{itemize}
\item[(a)] for each $i$, $q_{j,i}$ has $i$ distinct real roots;
\item[(b)] the roots of $q_{j,i}$ and $q_{j,i-1}$ are interlaced.
\end{itemize}

For $j < k+1$, the monic polynomial $q_j(t)$ of degree $j$
is orthogonal to all polynomials of degree at most $j-1$ with respect to the
signed measure $\mu_j(t)$ (this follows from the quadrature \eqref{quad-2-level}; see also the
paragraph just before Lemma 3.3 in \cite{CK}). Since such a polynomial is unique,
we conclude that it coincides with $q_{j,k-j+1}(t)$, i.e.
\[ q_{j,k-j+1}(t) = (t - \beta_1) \ldots (t - \beta_{k-j+1})=:q_{k-j+1}(t) \]
for every $j \leq k$. This and the fact that the roots are interlaced, imply, again as in \cite{CK}, that
for $i < k-j+1$, the largest root of $q_{j,i}(t)$ is less than $\beta_{k-j+1}$.
Therefore, $q_{j-1,i}(\beta_{k-j+2}) \neq 0$ for every $i \leq k-j+1$. Note that
in fact $q_{j-1,i}(\beta_{k-j+2})>0$ (we need this below).
Then there are constants
$\alpha_{j,i}$ such that for $i \leq k-j+1$,
\[ q_{j,i}(t) = \frac{q_{j-1,i+1}(t) + \alpha_{j,i}q_{j-1,i}(t)}{t-\beta_{k-j+2}}. \]

\begin{lemma}
\label{small-j}
For $1 \leq j \leq k+1$ and $i \leq k-j+1$, the polynomial $q_{j,i}$ is a positive linear
combination of the polynomials $q_{j-1,0},\ldots,q_{j-1,i}$.
\end{lemma}

\begin{proof}
We argue as in the end of the proof of Lemma \ref{thmj=0}. Define $d_0,\ldots,d_i$ so that
\[ q_{j,i}(t) =\sum_{\ell=0}^i d_\ell q_{j-1,\ell}(t). \]
For every $\ell \leq i$, we have by orthogonality
\[ \int_{-1}^1 \left(q_{j-1,i+1}(t) + \alpha_{i,j}q_{j-1,i}(t)\right) \cdot
\frac{q_{j-1,\ell}(t) - q_{j-1,\ell}(\beta_{k-j+2})}{t - \beta_{k-j+2}} d \mu_{j-1}(t) = 0, \]
since the polynomial $\frac{q_{j-1,\ell}(t) - q_{j-1,\ell}(\beta_{k-j+2})}{t - \beta_{k-j+2}}$
has degree $\ell-1 \leq i-1$. This and the formula for
$q_{j,i}(t)$ imply that
\[ \int_{-1}^1 q_{j,i}(t) q_{j-1,\ell}(t) d \mu_{j-1}(t) = q_{j-1,\ell}(\beta_{k-j+2}) \int_{-1}^1
q_{j,i}(t) d \mu_{j-1}(t). \]
Therefore
\[ d_\ell \int_{-1}^1 q_{j-1,\ell}(t)^2 d \mu_{j-1}(t) = d_0 q_{j-1,\ell}(\beta_{k-j+2})\int_{-1}^1
d \mu_{j-1}(t). \]
Because $\ell \leq i \leq k-j+1$, both integrals are positive. The largest root
of $q_{j-1,\ell}(t)$ is less than $\beta_{k-j+2}$, so $q_{j-1,\ell}(\beta_{k-j+2}) > 0$.
Thus, $d_0,\ldots,d_i$ all have the same sign, which is positive because $d_i>0$ (this
follows also from $d_0>0$ by the quadrature).
\end{proof}

We summarize the above work on the positive definiteness of the partial products $q_j(t)$.

\begin{theorem}
\label{small-j-pd} In the context of Lemma \ref{ExistIP},
all polynomials $q_{j}(t)= (t - \beta_1) \ldots (t - \beta_{j})$, $j=0,1,\ldots,k-2$, are positive definite. The polynomials
$q_{k-1}(t)$, $q_k(t)$, and $q_{k+1}(t)$ are $(1,0)$-positive definite if \eqref{pd-qk-1}, \eqref{c3cond2} and $c_i \geq 0$ are
fulfilled, respectively.
\end{theorem}

Now Theorem \ref{small-j-pd} and the Krein condition \eqref{krein} imply the following.

\begin{theorem}
\label{fhinanh}
In the context of Lemma \ref{ExistIP}, if \eqref{pd-qk-1} and  $c_i \geq 0$, $i=1,2,3$, hold, then the polynomial $f^h$ expands with nonnegative Gegenbauer coefficients; i.e., $f_i^h \geq 0$ for every $i$.
\end{theorem}

\begin{remark}
In general, the condition $c_i \geq 0$ is not necessary to verify the regular positive definiteness of $q_{k+1}, q_k, q_{k-1}$. However, in the extensive numerical computations we have not observed a case where $f^h$ is positive definite, while this condition fails.
\end{remark}

\subsection{Second level bounds on $\Lambda_{n,k}$.}

We can combine the subsections 3.1 and 3.2 into the following theorem extending Theorem \ref{thm3.2}, 
which we shall refer to as second level ULB.

\begin{theorem}
\label{fh-good}
In the context of Theorem \ref{QRthm} and Lemma \ref{ExistIP}, if \eqref{pd-qk-1} and $c_i \geq 0$, $i=1,2,3$, are satisfied, then the Hermite interpolant $f^h(t)$ defined by \eqref{PP-extr} belongs to the class $ \Lambda_{n,k} \cap A_{n,h}$. The following second level universal lower bound holds
\begin{equation}
\label{ULB-2}
\mathcal{E}_h(n,N) \geq S_\tau(n,N;h):=N^2\sum_{i=1}^{k+1} \theta_i h(\beta_i).
\end{equation}
Moreover, $f^h (t)$ is the unique optimal polynomial that yields (see \eqref{Wdef})
\[ N^2f_0^h-Nf^h(1)=\mathcal{W}_{h,\Lambda_{n,k}}(n,N)=S_\tau(n,N;h). \]
\end{theorem}

\begin{proof} Corollary \ref{Ineq} shows that $f^h \leq h(t)$ for all $t \in [-1,1)$ and Theorem \ref{fhinanh} says that $f^h$ is positive definite. 
Therefore $f^h \in A_{n,h}$, yielding immediately that $f^h \in \Lambda_{n,k} \cap A_{n,h}$. 

The calculation of the second level ULB \eqref{ULB-2} produced by $f^h$ is straightforward by the
$1/N$-quadrature rule \eqref{quad-2-level}. We have
\[ N^2f_0^h-Nf^h(1)=N^2\sum_{i=1}^{k+1} \theta_i f^h(\beta_i)=N^2\sum_{i=1}^{k+1} \theta_i h(\beta_i) \]
(we used $f^h(\beta_i)=h(\beta_i)$ from the interpolation). Now Theorem \ref{THM_subspace} shows the optimality of $f^h$. 
\end{proof}

The inclusion $P_\tau \subset \Lambda_{n,k}$ implies that $S_\tau(n,N;h) > R_\tau(n,N;h)$ (see also
the proof of the sufficiency of Theorem 4.1 in \cite{BDHSS}).

We proceed with explanation of the second level bound on $\mathcal{A}(n,s)$ improving on the first level; i.e. on the Levenshtein bounds.

Methods for obtaining better than the Levenshtein bounds utilizing polynomials of degrees $m+3$ and $m+4$ (in our terminology -- for finding second level bounds) were developed previously. For example, in \cite{OS} Odlyzko and Sloane discretized the constraint $f(t)\leq 0$ in $[-1,1/2]$ and applied the simplex method to target the so-called kissing number problem; in \cite{Boy1} Boyvalenkov proposed a computational method to approximate optimal polynomials of degree $\tau(n,N)+3$ and $\tau(n,N)+4$; and Lagrange multipliers were applied in some cases for utilization of the conditions $f_i=0$ by Nikova and Nikov in \cite{NN}.

Our second level bound for $\mathcal{A}(n,s)$ is obtained by the {\em second level Levenshtein-type polynomial}
\begin{equation}\label{LevPoly2}
g(t):=H_{\Lambda_{n,k}}(0;q_{k+1}q_k).
\end{equation}
Indeed, the conditions for existence and uniqueness of $g$ and its belonging to $\Lambda_{n,k} \cap B_{n,\beta_{k+1}}$
are the same as these for $f^h \in \Lambda_{n,k} \cap A_{n,h}$. 
Thus, $g(t)  \in \Lambda_{n,k} \cap B_{n,\beta_{k+1}}$ can be applied in Theorem \ref{thm-lp-sph} to give the second level bound on $\mathcal{A}(n,s)$.
Moreover, as the monotonicity of the Levenshtein bound implies the lower bound \eqref{snm-L} on the quantity $s(n,N)$, we obtain similarly
a second level bound on $s(n,N)$ which inproves on \eqref{snm-L}. The discussion of this paragraph is summarized in the next theorem.

\begin{theorem}
\label{lev_bound-imrovement}
In the context of Theorem \ref{QRthm} and Lemma \ref{ExistIP}, if \eqref{pd-qk-1} and $c_i \geq 0$, $i=1,2,3$, are satisfied, then
the polynomial $g(t)$ defined by \eqref{LevPoly2}  belongs to the class $\Lambda_{n,k} \cap B_{n,\beta_{k+1}}$. 
The following second level universal bound hold

\begin{itemize}
\item[{\rm (a)}] $\mathcal{A}(n,\beta_{k+1}) \leq g(1)/g_0=N<L_{2k-1}(n,\beta_{k+1})$,
\medskip

\item[{\rm (b)}] $s(n,N) \geq \beta_{k+1}>\alpha_k$.
\end{itemize}
Moreover, there exist no polynomials in $\Lambda_{n,k} \cap B_{n,\beta_{k+1}}$ which give better bound than $\mathcal{A}(n,\beta_{k+1}) \leq N$. 
\end{theorem}

\begin{proof} As discussed above, we have $g \in \Lambda_{n,k} \cap B_{n,\beta_{k+1}}$ and therefore $\mathcal{A}(n,\beta_{k+1}) \leq g(1)/g_0$.
Then the $1/N$-quadrature rule \eqref{quad-2-level} gives $g_0N=g(1)$; i.e., $\mathcal{A}(n,\beta_{k+1}) \leq N$. 
The strict monotonicity of the Levenshtein bound and 
$\alpha_k<\beta_{k+1}$ from Theorem \ref{IL} imply the strict inequality in (a) since $N=L_{2k-1}(n,\alpha_k)$. The optimality of $g(t)$
follows from \eqref{quad-2-level}. 

The proof of Theorem \ref{LevUpper} can be adapted to derive (b). It can be verified also whenever the ULB-space $\Lambda_{n,k}$ 
comes for an interval $[N_1,N_2] \subset \left(D(2k-1,n),D(2k,n)\right]$. 
\end{proof}

Note that the value $N$ in our new bound  $\mathcal{A}(n,\beta_{k+1}) \leq N$ is set well before we actually find what we do improve (i.e.,
the number $\beta_{k+1}$ and then the bound $L(n,\beta_{k+1})$). Since $\mathcal{A}(n,\beta_{k+1}) \leq \lfloor N \rfloor$, one would prefer to have  in this setting $N$ slightly less than an integer. 

\begin{remark}
Having defined the second level $1/N$-quadrature rule \eqref{quad-2-level} we can construct
second level test functions \eqref{TestFunctions} from Theorem \ref{THM_subspace}. Investigation of their signs will give,
as its first level counterpart does, necessary and sufficient conditions for existence of further
improvements by linear programming.
\end{remark}

%
%
%
%

\subsection{Lifting the Levenshtein framework, even case $\tau(n,N)=2k$ (sketch)}

The even case $\tau(n,N)=2k$ is quite similar. Let $n$ and $N$ be such that $Q_{2k+3}^{(n,N)}<0$ (the sign of
$Q_{2k+4}^{(n)}$ can be arbitrary). Now  the {\em skip-two/add-two} subspace \eqref{Lambdakn} is
\[ \Lambda_{n,k} = \mathcal{P}_{2k} \oplus \mbox{span} \left( P_{2k+3}^{(n)}, P_{2k+4}^{(n)}\right) \]
and our target is a $\Lambda_{n,k}$-LP-extremal polynomial
\[ f^h(t)=f_{2k+4}^hP_{2k+4}^{(n)}(t)+f_{2k+3}^hP_{2k+3}^{(n)}(t)+\sum_{i=0}^{2k} f_i^hP_i^{(n)}(t) \in \Lambda_{n,k} \cap A_{n,h}. \]
The $1/N$ quadrature rule exact for $\Lambda_{n,k}$ is
\[ f_0=\frac{f(1)}{N}+\theta_0f(-1)+\sum_{i=2}^{k+2} \theta_i f(\beta_i), \]
where $\beta_1=-1$ and the nodes $\beta_2,\beta_3,\ldots,\beta_{k+2}$ are the roots of the equation
\begin{equation}
\label{beti-2}
P_{k+1}^{1,1}(t)+d_1P_{k}^{1,1}(t)+d_2P_{k-1}^{1,1}(t)+d_3P_{k-2}^{1,1}(t)=0
\end{equation}
(compare to \eqref{def-alfi} with $\varepsilon=1$). The interlacing now is $\alpha_1=\beta_1=-1$ and
$\beta_i<\alpha_i<\beta_{i+1}$ for $i=2,3,\ldots,k+1$.

Exactly as in the Levenshtein framework for the first level, in the even case $\tau(n,N)=2k$ 
the strengthened Krein condition \eqref{s-krein} should be used instead of \eqref{krein}.

The second level Hermite interpolant to $h$ is
\[ f^h:=H_{\Lambda_{n,k}}(h;(\cdot+1)s_{k+1}^2) \in \Lambda_{n,k} \cup A_{n,h}, \]
where $s_{k+1}(t)$ is the polynomial in the LHS of \eqref{beti-2}. The Levenshtein bound
is improved as in Theorem \ref{lev_bound-imrovement}.

\section{Two special examples of the second level lift -- the $24$-cell and the $600$-cell on $\mathbb{S}^3$}

As our approach yields next-level necessary and sufficient conditions for existence of better bounds, it is particularly illustrative to consider the cases $(n,N)=(4,24)$ and $(n,N)=(4,120)$. Both fall in the case of subsection 3.1. Moreover, for these parameters there are prominent codes, namely the  $24$-cell and the $600$-cell, whose properties are widely investigated in the literature \cite{A1, A2, Bo, BoJr, CCEK, CK, Co1, Co2}.

\subsection{The $24$-cell.}

The $(4,24)$-codes take prominence in the literature (\cite{Boy1, OS, M}). In particular, the $24$-cell code (derived from the $D_4$ root system)
solving the kissing number problem \cite{M}, is suspected to be a maximal code, but is not
universally optimal (see \cite{CCEK}). In this case $\tau=5$, $k=3$, $\varepsilon=0$, the Levenshtein nodes and weights are approximately
$\{\alpha_1,\alpha_2,\alpha_3\}=\{-0.817352, -0.257597, 0.47495\}$,
$\{\rho_1,\rho_2,\rho_3\}=\{ 0.138436, 0.433999, 0.385897\}$. This defines
the corresponding $1/24$-quadrature rule \eqref{defin_qf}.

The first seven test functions associated with the Levenshtein $1/24$-quadrature rule \eqref{defin_qf}
are shown approximately in Table 1. Two of them, namely $Q_8^{(4, 24)}$ and $Q_9^{(4,24)}$, are negative.
\begin{table}[h!]
  \centering
  \caption{Approximations of the first seven non-zero test functions for the Levenshtein $1/24$-quadrature rule}

\vspace*{6mm}

  \label{tab:table1}
  \begin{tabular}{c|c|c|c|c|c|c}
    $Q_6^{(4,24)}$ & $Q_7^{(4,24)}$ & $Q_8^{(4,24)}$ & $Q_9^{(4,24)}$ & $Q_{10}^{(4,24)}$& $Q_{11}^{(4,24)}$& $Q_{12}^{(4,24)}$ \\
    \hline
    $0.0857$ & $0.1600$ & $-0.0239$&$-0.0204$&$0.0642$& $0.0368$ & $0.0598$\\
  \end{tabular}
\end{table}

\begin{figure}[ht]
\centering
\vspace{1mm}
\includegraphics[scale=.61]{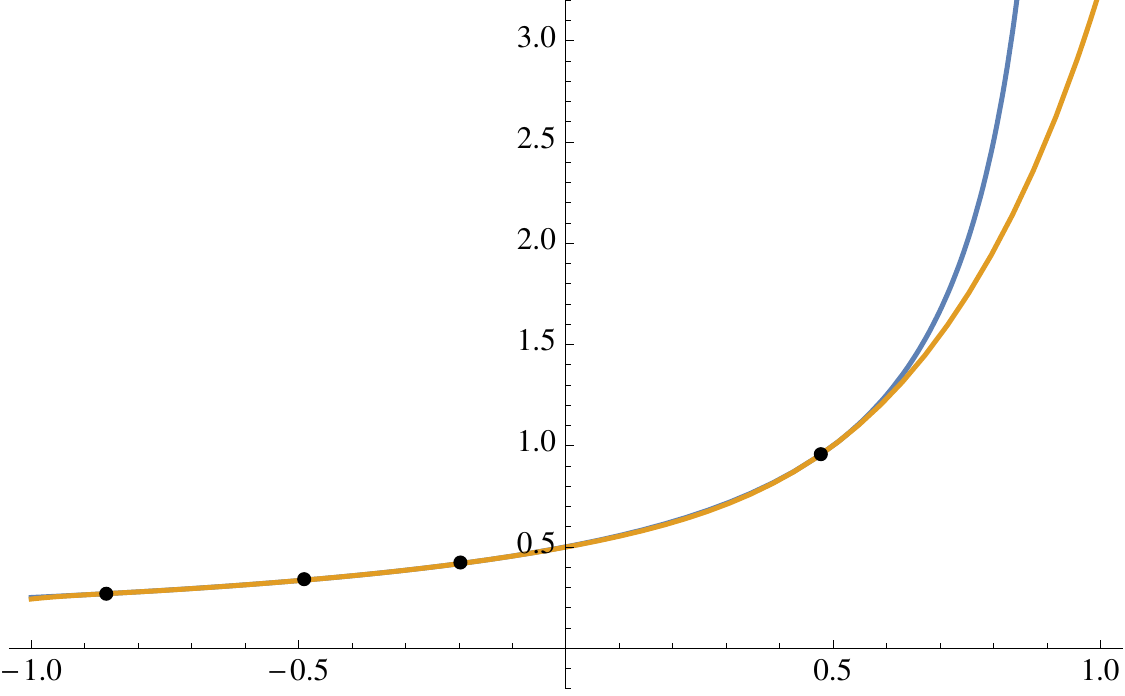}
\centering
\caption{The $(n,N)=(4,24)$ case optimal interpolant -- Newton potential}
\label{HermiteD4}
\end{figure}

The next assertion is a $(4,24)$-code version of Theorem \ref{THM_subspace}.

\begin{theorem} \label{Main} The collection of nodes and weights $\{(\beta_i, \theta_i)\}_{i=1}^{4}$
\begin{equation}
\label{beta-theta4-24}
\begin{split}
\{\beta_1,\beta_2,\beta_3, \beta_4 \}&= \{-0.86029..., -.0.48984..., -0.19572..., .0.478545...\}\\
\{ \theta_1, \theta_2, \theta_3,\theta_4 \}&=\{ 0.09960..., 0.14653..., 0.33372..., 0.37847...\},
\end{split}
 \end{equation}
define $1/24$-quadrature rule \eqref{quad-2-level} that is exact for the subspace
$\Lambda_{4,3}:=\mathcal{P}_5 \oplus {\rm span}\left(P_8^{(4)},P_9^{(4)}\right)$.
For every absolutely monotone $h$ the Hermite interpolant
$f^h(t)=H_{\Lambda_{4,3}}(h;q_4^2)$ exists and belongs to $\Lambda_{4,3} \cap  A_{4,h}$. Subsequently, $\Lambda_{4,3}$ is a ULB-space and
the following  universal lower bound (and an improvement of \eqref{bound_odd}) holds
\[ \mathcal{E}_h(4,24)\ge  S_5(4,24;h)=24^2\sum_{i=1}^{4} \theta_i h(\beta_i). \]
The test functions $Q_j^{(4,24)}$ associated with the second-level $1/24$-quadrature rule \eqref{quad-2-level} with nodes and weights
\eqref{beta-theta4-24} are positive for all $j \in \mathbb{N} \setminus \{1,2,3,4,5,8,9\}$ and therefore $f^h(t)$ is LP-optimal {\rm(}see Figure 1{\rm)}.
\end{theorem}

\begin{proof}
Following the procedure in subsection 3.1 we arrive at the equations \eqref{duad-forL-1}-\eqref{duad-forq}:
\[ 6776 c_1^2 c_3-2200 c_1 c_2 - 2233 c_1 c_3 - 7425 c_2 c_3+875c_1+675c_2-1134 c_3=0, \]
\[ 195657 c_1^2 c_3 - 698775 c_1 c_2 c_3 - 63525 c_1 c_2 - 83006 c_1 c_3 +226875 c_2^2
+889350 c_3^2 +61600 c_2 +76538 c_3 +36750 = 0, \]
\[ 77c_1-275 c_2-1463 c_3+217=0, \]
respectively. We obtain
\begin{equation} \label{c's24_4} c_1 = 0.909977..., \ c_2=0.501716..., \ c_3= 0.101911... \end{equation}
Resolving the equation $q_4(t)=0$ we get distinct zeros $\beta_1,\beta_2,\beta_3, \beta_4$, all in $(-1,1)$. The weights $\theta_1, \theta_2, \theta_3,\theta_4$ are all positive, hence by Theorem \ref{QRthm} we obtain that \eqref{beta-theta4-24} defines the unique $1/24$-quadrature on $\Lambda_{4,3}$  (observe the interlacing of the first and second level quadrature nodes as described in Theorem \ref{IL}).

The conditions of Theorem \ref{implications-by-pps} are also checked directly with the interpolation set
$T= \left\{\beta_1, \beta_1, \ldots, \beta_4, \beta_4\right \}$. With the constants from \eqref{c's24_4} we have that \eqref{existence-cond} holds, which implies the existence and uniqueness of the Hermite interpolant $H_{\Lambda_{4,3}} (h;q_4^2 )$ on the subspace $\Lambda_{4,3}$. Computing $(A_1,B_1) \approx (1.2197,-1.7419)$ and $(A_2,B_2) \approx (1.5983,-2.7379)$ yields that condition \eqref{ineq-check-a1b1} holds so the inequality $H_{\Lambda_{4,3}} (h;q_4^2 ) \leq h(t)$ follows from Corollary \ref{Ineq}. We can verify the positive definiteness of the partial products directly or via Theorem \ref{small-j-pd}. Theorem \ref{fhinanh} extends this property to $H_{\Lambda_{4,3}} (h;q_4^2 ) $.Therefore,  $H_{\Lambda_{4,3}} (h;q_4^2 ) \in \Lambda_{4,3} \cap A_{4,h}$ for every absolutely monotone $h$.

The signs of the second level test functions $Q_j^{(4,24)}$, $j\in \mathbb{N} \setminus \{1,2,3,4,5,8,9\}$, can be investigated as in \cite[Section 4.3]{BDHSS}.
In this case the parameter $j_0=j_0(4,24)$ from that article is equal to $15$.
\end{proof}

The typical behaviour of the second level energy bounds with respect to the first level ULB and actual energies is well illustrated by the situation with
$h(t)$ being the Newton potential. In this case the first level energy bound is 333, the second level is $\approx 333.15$ and the best known energy 
is 334 -- the energy of the 24-cell \cite{BBCGKS}. The second level bound on maximal codes is better understood by observing that starting with 
$24-\epsilon$, where $\epsilon>0$ is very small will result in a second level bound $\mathcal{A}(4,s) \leq 23$ instead of $\mathcal{A}(4,s) \leq 24$, where $s$ is 
slightly smaller than $\beta_4$. With a few exceptions, all second level results for integer $N=\mathcal{A}(n,\beta_{k+1})$ can be treated this way.

The LP-optimality of the second level polynomial implies that the $24$-cell can not be shown to be optimal for any particular absolutely monotone potential by linear programming. In fact, it was proved in \cite{CCEK} that the $24$-cell is not universally optimal. 

We note also that the second level Levenshtein polynomial \eqref{LevPoly2} for $(n,N)=(4,24)$ (along with $(n,N)=(4,25)$) was derived first by Arestov and Babenko in \cite{AB}.

\subsection{The $600$-cell } 

In the case $\tau(4,120)=11$ the first level ULB is provided by the Hermite interpolant to the potential function at the Levenshtein nodes
$(\alpha_i)_{i=1}^6 \approx (-0.9356, -0.7266, -0.3810, 0.04406, 0.4678, 0.8073)$. Since
$Q_{12}^{(4,120)}>0$, $Q_{13}^{(4,120)}>0$, $Q_{14}^{(4,120)}<0$, and $Q_{15}^{(4,120)}<0$,
we apply the technique from Section 3.1.

To prove that $\Lambda_{4,6}=\mathcal{P}_{11} \oplus {\rm span} \left(P_{14}^{(4)}, P_{15}^{(4)}\right)$ is a ULB-space we proceed similarly to subsection 4.1. The equations  \eqref{duad-forL-1}-\eqref{duad-forq} are too long to be stated here. We find
\[ (c_1,c_2,c_3)=(0.944...,0.532..., 0.318...). \]
This leads us to a $1/120$-quadrature with $7$ internal nodes
\[ (\beta_i)_{i=1}^7=(-0.9819...,-0.7965...,-0.4765..., -0.1654..., 0.0977..., 0.4754..., 0.8079...) .\]
The corresponding weights $\{\theta_i\}_{i=1}^7$ are all positive, and by Theorem \ref{QRthm} we conclude the $1/120$-quadrature is exact in the
space $\Lambda_{4,6}$. In a similar fashion as in subsection 4.1 we verify that the Hermite interpolant to the potential
function $h$ at these nodes exists, stays underneath the potential, and is positive definite, i.e. $H_{\Lambda_{4,6}}(h;q_7^2)\in \Lambda_{4,6} \cap A_{4,h}$. Actually, the new test functions of order $12$ and $13$ are positive and hence, $H_{\Lambda_{4,6}}(h;q_7^2)$ is $\mathcal{P}_{15}$-LP-optimal.

We note that with the Newton potential the first level ULB is $10786.8$ while the second level gives $10788.2$. This is still below the
actual Newton energy $10790$ of the 600-cell which is going to be achieved by a third level lift in the next section.  

\section{The universal optimality of the 600-cell revisited - third level lift}

\subsection{Universal optimality of the $600$-cell.}
We start with a general result about
generation of $1/N$-quadrature rules by (good) codes. We define the $i$-th moment of a spherical code $C=\{x_1,x_2,\ldots,x_N\}$ by
 \begin{equation}\label{Mk0}
 M_i(C):=\sum_{j,\ell=1}^N P^{(n)}_i(\langle x_j , x_\ell\rangle ). \nonumber\end{equation}
It is well known that $M_i(C) \geq 0$ with equality if and only if
$\sum_{j=1 }^N Y(x_j)=0$ for all spherical harmonics $Y\in\mathrm{Harm}(i)$. The set
\[ \mathcal{I}(C):=\{i \in \mathbb{N}\colon M_i(C)=0\} \]
is called the {\em index set} of $C$.  Hence, $C$ is a spherical $\tau$-design if and only if  $\{1, 2, \ldots, \tau\}\subseteq \mathcal{I}(C)$.

It is straightforward to see that the identity
\begin{equation}
\label{moments-energy}
E_f(C)=f_0N^2-f(1)N+\sum_{i=1}^r f_i M_i(C)
\end{equation}
holds for any polynomial $f(t)=\sum_{i=0}^r f_i P_i^{(n)}(t)$.

A key component of the proof of Theorem \ref{600cellThm} below is that an $N$-point code $C \subset \mathbb{S}^{n-1}$
provides a $1/N$-quadrature rule  that is exact on the subspace  spanned by $P_i^{(n)}$ for $i$ in the index set  $\mathcal{I}(C)$.

\begin{theorem}\label{codeQR}
Let $C \subset \mathbb{S}^{n-1}$ be an $N$-point code and
\begin{equation}\label{alphacode}
\{-1\le \alpha_1< \cdots < \alpha_m<1\}:=\left\{\langle x,  y \rangle\colon x \neq y \in C\right\},
\end{equation}
be the set of inner products with
\begin{equation}\label{rhocode}
\rho_\ell:=\frac {\big|\{(i,j) \colon \langle x_i,x_j\rangle=\alpha_\ell\}\big|}{N^2}, \qquad \ell=1,\ldots, m,
\end{equation}
the relative frequency of occurrence of  $\alpha_\ell$. If
$$
\Lambda(C):={\rm span}\{1, P_j^{(n)}(t) \colon j \in \mathcal{I}(C)\},
$$
then
 $\{(\alpha_\ell, \rho_\ell)\}_{\ell=1}^m$  is a $1/N$-quadrature rule  exact for  $\Lambda(C)$ and for any $f\in \Lambda(C)$
 \begin{equation}\label{Efcode}
 E_f(C)= N^2\sum_{\ell=1}^{m}\rho_\ell f(\alpha_\ell)=N^2(f_0-f(1)/N).
 \end{equation}
 \end{theorem}
\begin{proof}
Suppose $f\in \Lambda(C)$ is a polynomial of the form $f=\sum_{j=0}^n f_i P_i^{(n)}$ with
$f_i \neq 0$ if and only if $i \in\mathcal{I}(C) \cup \{0\}$.  The first equality in \eqref{Efcode}
holds from the definitions \eqref{alphacode} and \eqref{rhocode} (in fact for any function
$f \in \Lambda(C)$).  The second equality in \eqref{Efcode} follows from \eqref{moments-energy}
and also shows that $\{(\alpha_\ell, \rho_\ell)\}_{\ell=1}^m$ is exact for $\Lambda(C)$.
\end{proof}
We now turn to a third level lift and apply it to derive an alternative proof of the $600$-cell $W_{120}$ universal optimality. The $16$-th degree second level test function associated with
the new nodes being negative prompts us to seek a $1/120$-quadrature with the parameters of the 600-cell:
eight nodes
$$(\gamma_i)_{i=1}^8=\left\{-1,\frac{-1-\sqrt{5}}{4},-\frac{1}{2},
  \frac{1-\sqrt{5}}{4} ,0, \frac{\sqrt{5}-1}{4},\frac{1}{2}, \frac{1+\sqrt{5}}{4}\right\}$$
occurring with corresponding relative frequencies (weights)
$$\{\nu_1,\dots,\nu_8\}=\left\{\frac{1}{120},\frac{1}{10},\frac{1}{6},\frac{1}{10}
   ,\frac{1}{4},\frac{1}{10},\frac{1}{6},\frac{1}{10}\right\}.$$

By direct computation (or see \cite[Section 3]{AY}, \cite[Theorem 5.1]{BDK}) one may verify
that the index set of $W_{120}$ contains $\{1,2, \ldots, 19\}\setminus\{12\}$ and
$\{(\gamma_i,\nu_i)\}_{i=1}^8$ form a (third level) $1/120$-quadrature rule exact
on the subspace $\mathcal{P}_{19}\cap\{P_{12}^{(4)}\}^\perp$.
The problem is then to establish that for a given $h$ absolutely monotone on $[-1,1]$ there exists some
positive definite polynomial $f$ of degree at most 19 with $f_{12}=0$ and such that $h(t)\ge f(t)$, $t\in [-1,1)$,
with equality if $t\in   \{\gamma_1,\dots,\gamma_8\}$.  Cohn and Kumar \cite{CK} consider the subspace
\[ \Lambda_3:=\mathcal{P}_{10} \oplus \mbox{span} \left(P_{14}^{(4)},P_{15}^{(4)},P_{16}^{(4)},P_{17}^{(4)}\right) \]
and show there is a unique $f\in \Lambda_3$
to the interpolation problem $f(-1)=h(-1)$, $f(\gamma_i)=h(\gamma_i)$, and
$f'(\gamma_i)=h'(\gamma_i)$ for $i=2,3, \ldots 8$ and, furthermore, that $f$ is positive definite and stays
below $h$; i.e., that $f\in A_{4,h} \cap\Lambda_3$. We find two other subspaces, namely,
\[ \Lambda_1:= \mathcal{P}_{10} \oplus \mbox{span} \left(P_{13}^{(4)},P_{14}^{(4)},P_{15}^{(4)},P_{16}^{(4)},P_{17}^{(4)}\right) \]
and
\[ \Lambda_2:=\mathcal{P}_{11} \oplus \mbox{span} \left(P_{14}^{(4)},P_{15}^{(4)},P_{16}^{(4)},P_{17}^{(4)}\right) \]
either of which yields a simpler proof of the universal optimality of $W_{120}$.

\begin{theorem}\label{600cellThm}
The 600-cell $W_{120}$ is universally optimal.  If $h$ is strictly absolutely monotone in $[-1,1]$ and $C_{120}$
is any 120-point code on $\mathbb{S}^3$ not isometric to $W_{120}$, then $$E_h(W_{120}) <E_h(C_{120}).$$
\end{theorem}

\begin{proof}
Let   $T= \left\{\gamma_1, \gamma_1, \ldots, \gamma_8, \gamma_8\right \}=\{t_0,t_1,\ldots,t_{15}\}$ and $g(t)=\prod_{i=1}^8(t-\gamma_i)^2$.
For $i=1$ or $2$, we  define 
\begin{equation} \label{PTjLamdef}
p_j(\Lambda_i,T;t):=g_j (t) +(A_j^i+B_j^it)g(t), \qquad j=0,1, \ldots, 15, 
\end{equation}
where $g_j(t)$ is defined in \eqref{partial-1} and $A_j^i$ and $B_j^i$ are the unique values so that $p_j(\Lambda_i,T;\cdot)$ is orthogonal to $P_{11}^{(4)}$
and $P_{12}^{(4)}$ if $i=1$ and to $P_{12}^{(4)}$ and $P_{13}^{(4)}$ if $i=2$.
Note that  $A_j^i=B_j^i=0$ for $j=0, \ldots, 10$ in the case $i=1$ and for $j=0, \ldots, 11$ in the case $i=2$. Observe that we have $p_j(\Lambda_i,T;t)=H_{\Lambda_i}(g_j;g)$.

By explicit computation aided by a computer algebra system (CAS), we compute the nonzero values of
$A_j^i$ and $B_j^i$  shown in Table~\ref{ABtable600}.  Observing (see \eqref{ineq-check-a1b1})
that $B_j^i\ge 0$ and $A_j^i+B_j^i\le 0$
for all $0\le j\le 15$ and $i=1,2$ shows that $(A_j^i+B_j^it)\le 0$ for $t\in [-1,1]$ and so
\begin{equation} \label{pjLam<pj}
p_j(\Lambda_i,T;t)\le g_j (t)  \qquad j=0,1, \ldots, 15,
\end{equation}
for $t\in [-1,1]$.

\begin{table}
\caption{The nonzero values of $A_j^k$ and $B_j^k$ in \eqref{PTjLamdef}.}
\label{ABtable600}
\centering
\begin{tabular}{|c|c|c|c|c|}
\hline
\rule{0pt}{.2in} $k $ & $ A_k^{1} $ & $ B_k^{1}$ & $ A_k^{2} $ & $ B_k^{2}$ \\[.05in]
\hline
\rule{0pt}{.25in} $11 $ & $ -\frac{128}{13} $ & $ \frac{352}{39} $ & $ 0 $ & $ 0$ \\
\rule{0pt}{.25in}$ 12 $ & $ -\frac{4\left(87+16 \sqrt{5}\right)}{13}  $ & $
   \frac{16\left(59+11 \sqrt{5}\right) }{39} $ & $ -\frac{220}{29}
   $ & $ \frac{192}{29}$ \\
\rule{0pt}{.25in}$ 13 $ & $ -\frac{2\left(210+79 \sqrt{5}\right)}{13}  $ & $
   \frac{4\left(279+107 \sqrt{5}\right)}{39}  $ & $ -\frac{2\left(234+55 \sqrt{5}\right)}{29}
   $ & $ \frac{16\left(25+6
   \sqrt{5}\right)}{29}  $\\
\rule{0pt}{.25in} $14 $ & $ \frac{ -725-301 \sqrt{5} }{26}  $ & $
   \frac{4\left(235+99 \sqrt{5}\right) }{39} $ & $ \frac{-965-413 \sqrt{5}}{58}
   $ & $ \frac{16\left(25+11
   \sqrt{5}\right)}{29}  $\\
\rule{0pt}{.25in} $15 $ & $ \frac{-471-185 \sqrt{5}}{52}   $ & $
   \frac{271+115 \sqrt{5}}{39}   $ & $ \frac{-983-345 \sqrt{5}}{116}
  $ & $ \frac{2\left(93+35
   \sqrt{5}\right)}{29}  $\\[.07in]
     \hline
\end{tabular}
\end{table}

Exact CAS computations show  that the coefficients  in the Gegenbauer expansion of
$p_j(\Lambda_i,T;t)$ are non-negative and therefore $p_j(\Lambda_i,T;t)$ is positive
semi-definite for $j=0,1, \ldots, 15$ and $i=1,2$.    Let $h$ be absolutely monotone on $[-1,1)$
and, for $i=1$ or $2$, let
\[ H_{\Lambda_i}(h;g) :=\sum_{j=0}^{15}h[t_0,t_1,\ldots, t_j]p_j(\Lambda_i,T;t). \]
By \eqref{pjLam<pj} and the non-negativity of the divided differences $h[t_0,t_1,\ldots, t_j]$, we have $H_{\Lambda_i}(h;g) \le H(h;g)(t)$ for $t\in[-1,1]$.
Additionally, we may use the remainder formula for the Hermite interpolation to write
$$
h(t)-H(h;g)(t)=h[\gamma_1,\gamma_1,\ldots, \gamma_8,\gamma_8,t]\prod_{i=1}^8(t-\gamma_i)^2\ge 0.
$$
showing $H(h;g)(t)\le h(t)$ for $t\in[-1,1)$.   We have therefore established that
$H_{\Lambda_i}(h;g) \in A_{4,h} \cap\Lambda_{i}$, $i=1,2$, and since
$H_{\Lambda_i}(h;g)(\gamma_j)=H(h;g)(\gamma_j)=h(\gamma_j)$ for $j=1,\ldots,8$ and $i=1,2$,
it follows from Theorem~\ref{THM_subspace}, Theorem~\ref{codeQR}, and the above discussion
concerning $H_{\Lambda_i}(h;g)$ that $E_h(W_{120})=\mathcal{E}_h(4,120)$ and therefore that $W_{120}$ is universally optimal.
\end{proof}

\begin{figure}[htbp]
\centering
\raisebox{.2in}{\includegraphics[width=2.7in]{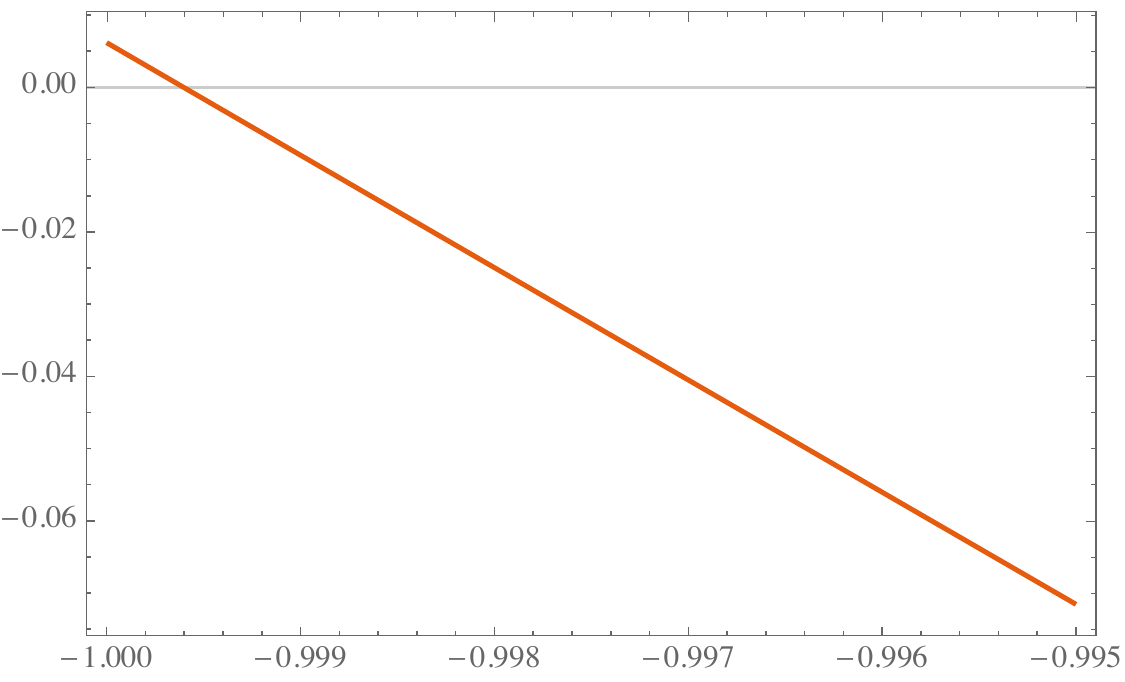}}\hspace{-.1in} \raisebox{1.43in}{\includegraphics[width=.6in]{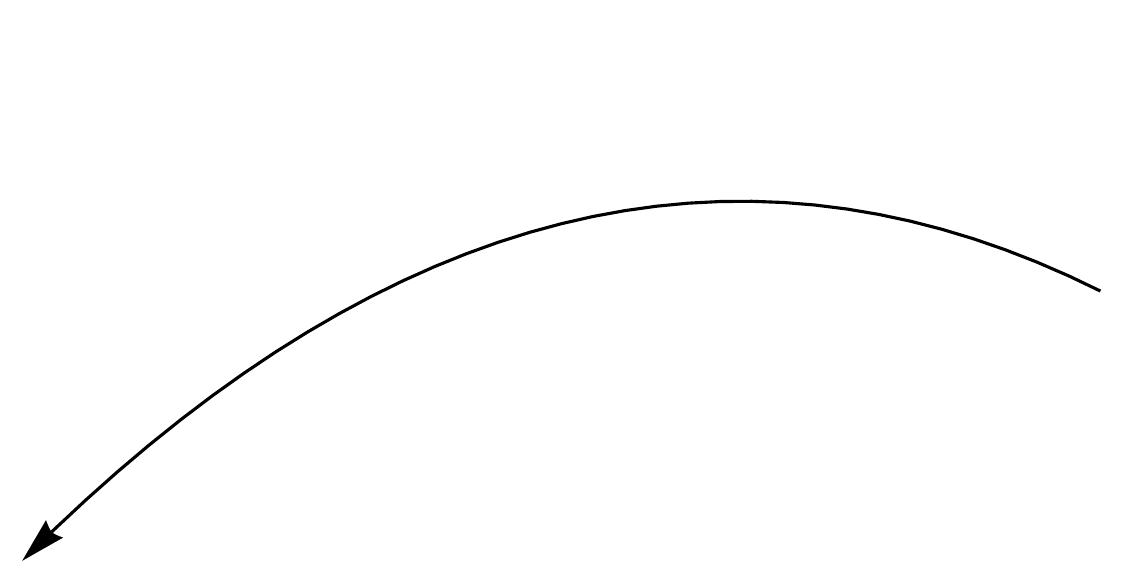}}\hspace{-.1in}
\includegraphics[width=2.7in]{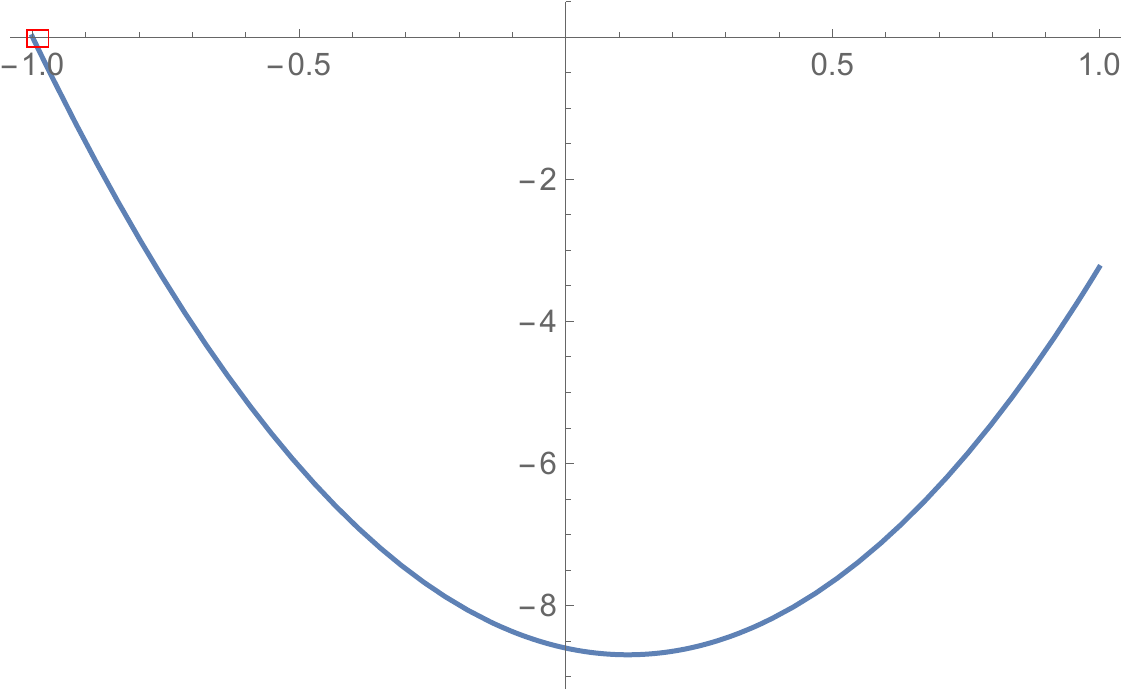}
\caption{Right: plot of $A_{14}^0+B_{14}^0t+C_{14}^0t^2$ on $[-1,1]$.  Left: Blowup of graph for $t\in [-1,-.995]$.}
\label{CKFig}
\end{figure}

It is interesting to consider why this approach does not work for $\Lambda_3$.   In this case, the interpolation
set $T_0=\left\{\gamma_1, \gamma_2,\gamma_2, \ldots, \gamma_8, \gamma_8\right \}$  has $\gamma_1=-1$ with multiplicity 1
and so the regular interpolant is a polynomial of degree at most 14.  Since there are 3 orthogonality conditions, we consider
\begin{equation} \label{PTjLamdef3}
p_j(\Lambda_3,T_0;t):=p_j(T_0;t)+(A_j^0+B_j^0t+C_j^0t^2)(t+1)\prod_{i=2}^8(t-\gamma_i)^2, \qquad j=0, 1, \ldots, 14,
\end{equation}
and again compute $A_j^0$, $B_j^0$, and $C_j^0$ using exact CAS computations. In this case we verify that
$p_j(\Lambda_3,T_0;t)$ for all $j=0,\ldots,14$ are positive semi-definite and that $p_j(\Lambda_3,T_0;t)\le p_j(T_0;t)$
for $j=0, \ldots, 13$.  For $j=14$ we have $A_{14}^0=\frac{ -27-11 \sqrt{5} }{6}$, $B_{14}^0=\frac{  -27-\sqrt{5} }{18}$,
and  $C_{14}^0=\frac{4\left(3+\sqrt{5}\right)}{3}$ and the quadratic term
$ A_{14}^0+B_{14}^0t+C_{14}^0t^2$ has a zero at $t=t^*=\frac{1}{48} \left(19-6 \sqrt{5}-\sqrt{2413+204
   \sqrt{5}}\right)\approx -.999603$ and $p_{14}(\Lambda_3,T_0;t)>0$ for $t\in [-1,t^*)$ as shown in Figure~\ref{CKFig}.
Cohn and Woo \cite{CW} proceed  by replacing the quadratic term in \eqref{PTjLamdef3} for $j=14$ with a cubic
term resulting in a degree eighteen generalized partial product that is positive semi-definite and negative on $[-1,1]$.

\subsection{Third level lift -- quadrature nodes of the $600$-cell.} It is noteworthy to say that as the $600$-cell is almost a $19$-design, its inner products $\{\gamma_i\}$ and relative frequencies $\{\nu_i\}$ form a (third level) $1/120$-quadrature rule exact on the subspace $\mathcal{P}_{19}\cap\{P_{12}^{(4)}\}^\perp$. Therefore, unlike the second level lift, we din't need to perform the necessary work to determine the quadrature nodes in subsection 5.1. Below we sketch briefly an adaptation of the computational framework described in subsection 3.1 to determine the quadrature nodes for the third level lift in the $(n,N)=(4,120)$ case. Namely, we seek the eight nodes $\{ \delta_i \}_{i=1}^8$ as roots of a polynomial (compare with \eqref{rk+1})
\[ q_8 (t) = P_{8}^{1,0}(t)+c_1P_{7}^{1,0}(t)+c_2P_{6}^{1,0}(t)+c_3P_{5}^{1,0}(t)+c_4P_{4}^{1,0}(t)+c_5P_{3}^{1,0}(t). \]
Similar to \eqref{Poly}, we require that
\[ L(t)=(1-t)q_{8}(t)\left(d_0P_{8}^{1,0}(t)+d_1P_{7}^{1,0}(t)+d_2P_{6}^{1,0}(t)+d_3P_{5}^{1,0}(t)+d_4P_{4}^{1,0}(t)+d_5P_{3}^{1,0}(t)\right) \in \Lambda_2, \]
where we use the orthogonality conditions
\[\langle L, P_{12}^{(n)} \rangle = \langle L, P_{13}^{(n)} \rangle =0\]
to express
\[d_4=d_4(d_0,d_1,d_2,d_3), \quad d_5=d_5(d_0,d_1,d_2,d_3).\]

Substituting $(d_0,d_1,d_2,d_3)$ with $(1,0,0,0)$, $(0,1,0,0)$, $(0,0,1,0)$, and $(0,0,0,1)$, respectively, we obtain four equations of degree at most $3$ in the variables $c_0, c_1, c_2, c_4, c_5$. Together with the linear equation (see \eqref{duad-forq})
\[
I_{8}+c_1I_7+c_2I_{6}+c_3I_{5}+c_4I_{4}+c_5I_{3}=\frac{1+c_1+c_2+c_3+c_4+c_5}{120},
\]
we find
\[ c_1=0.8947\dots, c_2=0.7894.\dots, c_3=0.6842\dots, c_4=0.2315\dots, c_5=0.1894\dots.\]
As expected, the roots $\{\delta_i\}_{i=1}^8$ of $q_8 (t)$ coincide with the inner products of the $600$-cell.

If we require that $L(t)\in \Lambda_1$, in which case the orthogonality conditions used are
\[\langle L, P_{11}^{(n)} \rangle = \langle L, P_{12}^{(n)} \rangle =0,\]
we arrive again at the same nodes. 

\subsection{Characterization of LP-optimal polynomials of minimal degree.} We conclude this section with a characterization of all polynomials
$f\in \mathcal{P}_{17}$ that are LP-optimal for $(n,N)=(4,120)$ and a given $h$ absolutely monotone on $[-1,1]$.
For $i=1,2$,  let $f_{h,\Lambda_i}(t)$ denote the LP-optimal polynomial in $\Lambda_i$ constructed  above  and let
$f_{h,\Lambda_3}(t)$ be the  LP-optimal polynomial whose existence is proved in  \cite{CK}.

\begin{figure}[htbp]
\centering
\includegraphics[width=3.in]{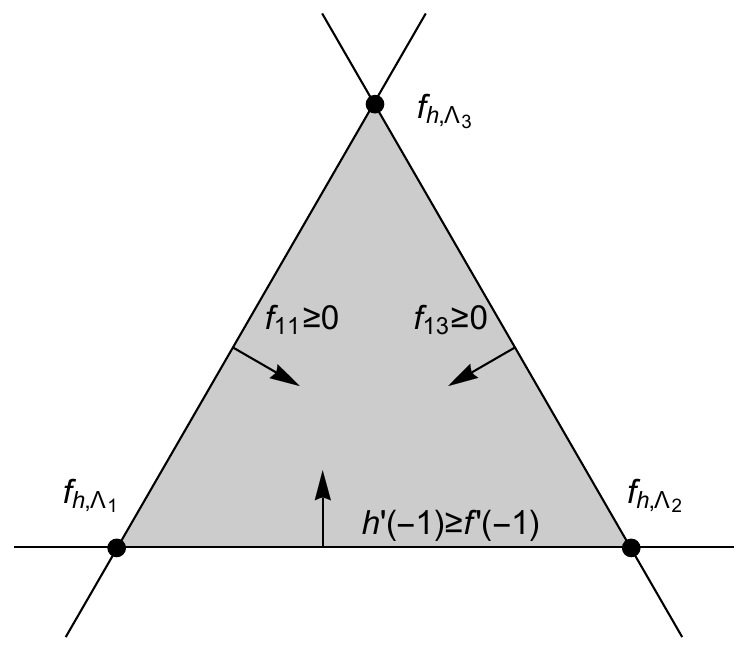}
\caption{Set of LP-optimal polynomials for 600-cell and given $h$ as in Theorem~\ref{TriThm}.}
\label{fig:s1}
\end{figure}

\begin{theorem}\label{TriThm}
Let $h$ be an absolutely monotone function on $[-1,1)$. Then $f\in \mathcal{P}_{17}$ is LP-optimal for $(n,N)=(4,120)$
if and only if $f$ is a convex combination of $f_{h,\Lambda_3}$, $f_{h,\Lambda_1}$, and $f_{h,\Lambda_2}$.
\end{theorem}

\begin{remark} In his proof that the $600$-cell $W_{120}$ is a maximal code \cite{A1}, Andreev utilizes the counterpart of Figure 4 triangle's (polynomial) vertex $f_{h,\Lambda_2}$ that corresponds to the maximal cardinality problem \eqref{ANS} and notes that the counterpart of the whole segment $h'(-1)=f'(-1)$ works. In a subsequent article \cite[Theorem 2]{A2} Andreev also proves that $W_{120}$ minimizes the Newton energy among all configurations of $120$ points by using a polynomial that lies in the interior of the side of the triangle (with $h(t)=1/(1-t)$) determined by the condition $h'(-1)=f'(-1)$.
\end{remark}

\begin{proof}
  Theorem~\ref{600cellThm}  implies that any LP-optimal $f$ in $ \mathcal{P}_{17}$, $(n,N)=(4,120)$, must satisfy the necessary conditions:
\begin{enumerate}
\item[\rm (a)]  $f_{12}=0$ and $f_j \ge 0,$  $j=0, 1,\ldots, 17$,
\item[\rm (b)] $f(\gamma_j)=h(\gamma_j)$,   $j=1, 2, \ldots, 8$,
\item[\rm (c)]  $f'(-1)\le h'(-1)$ and $f'(\gamma_j)=h'(\gamma_j)$, \,  $j=2, 3, \ldots, 8.$
\end{enumerate}
Suppose $f,g\in \mathcal{P}_{17}$ both satisfy conditions (a), (b), and (c). The equality constraints in (b) and (c) imply
$$
f(t)-g(t)=(t+1)\prod_{j=2}^8(t-\gamma_j)^2(A+Bt+Ct^2),
$$
for some constants $A$, $B$, $C$.  Further, from $f_{12}=g_{12}=0$, a direct  computation gives
$5A=-6(B+C)$   and so
\begin{equation}\label{fBC}
f(t)-g(t)=(t+1)\prod_{j=2}^8(t-\gamma_i)^2(B(t-6/5)+C(t^2-6/5)).
\end{equation}
Then, exact computations give
\begin{equation*}\begin{split}
f'(-1)-g'(-1)&=\lambda_{1,2}(B,C):=-\frac{15 (113 B+83 C)}{4096},\\
\quad (f-g)_{11}&=\lambda_{1,3}(B,C):=-\frac{\pi  (12 B+7 C)}{2621440}, \\
\quad (f-g)_{13}&=\lambda_{2,3}(B,C):=\frac{\pi  (11 C-24   B)}{18350080}.
\end{split}\end{equation*}
Let $g:=f_{h,\Lambda_1}$ and $f_{B,C}$ denote the polynomial defined by \eqref{fBC} for given $B$ and $C$. Also, let $\alpha:=g_{13}=(f_{h,\Lambda_1})_{13}$.  Since $g'(-1)=h'(-1)$ and $g_{11}=0$, it follows that if
 $f_{B,C}$ is optimal, then $(B,C)$ must lie in the intersection $\Delta$ of the half-spaces  $\lambda_{1,2}(B,C)\le 0$, $\lambda_{1,3}(B,C)\ge 0$, and $\lambda_{2,3}(B,C)\ge -\alpha$.  Let $L_{1,2}$, $L_{1,3}$, and $L_{2,3}$ denote the lines
 $\lambda_{1,2}(B,C)=0$, $\lambda_{1,3}(B,C)=0$, and $\lambda_{2,3}(B,C)=-\alpha$, respectively.  Let $(B_1,C_1)$, $(B_2,C_2)$, $(B_3,C_3)$ be the intersection points $\{(B_1,C_1)\}=\{(0,0)\}=L_{1,2}\cap L_{1,3}$, $\{(B_2,C_2)\}=L_{1,2}\cap L_{2,3}$, and $\{(B_3,C_3)\}=L_{1,3}\cap L_{2,3}$ and observe   that $f_{B_k,C_k}=f_{h,\Lambda_k}$.    Then $f_{B,C}$ is a convex combination of $f_{h,\Lambda_3}$, $f_{h,\Lambda_1}$, and $f_{h,\Lambda_2}$ if and only if $(B,C)\in \Delta$.   Since any LP-optimal polynomial $f\in \mathcal{P}_{17}$, $(n,N)=(4,120)$, is of the form $f=f_{B,C}$, the proof is completed.
\end{proof}

 We remark that if $\alpha=0$ in the above proof then $\Delta=\{(0,0)\}$ and   $f_{h,\Lambda_1}=f_{h,\Lambda_2}=f_{h,\Lambda_3}$
is the only LP-optimal polynomial in $ \mathcal{P}_{17}$, $(n,N)=(4,120)$.
Otherwise, if $\alpha>0$ then $\Delta$ forms a non-degenerate triangle.

\section{Numerical illustration of the second level ULB}

In this section we give examples of second level bounds which illustrate well the typical behaviour in the framework of Section 3.

Figure 4 illustrates the first and second level ULB for $\mathcal{E}_h(12,N)$ (on the left) for Newton potential, and the first and second level bounds for $\mathcal{A}(12,s)$ (on the right) in the particular case $(n,\tau)=(12,9)$. 

\begin{figure}[ht]
\centering
\includegraphics[scale=.69]{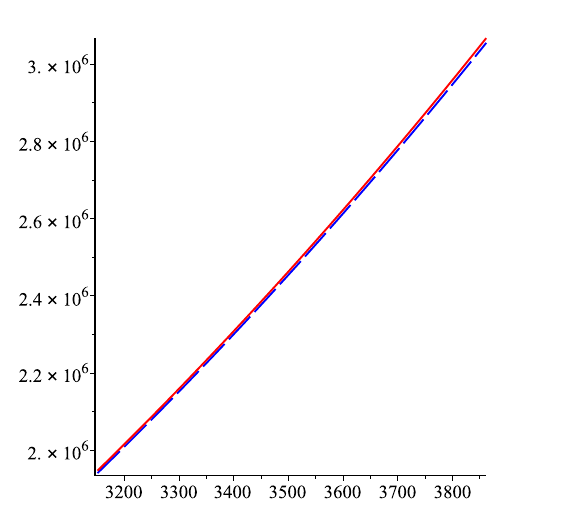}
\includegraphics[scale=.69]{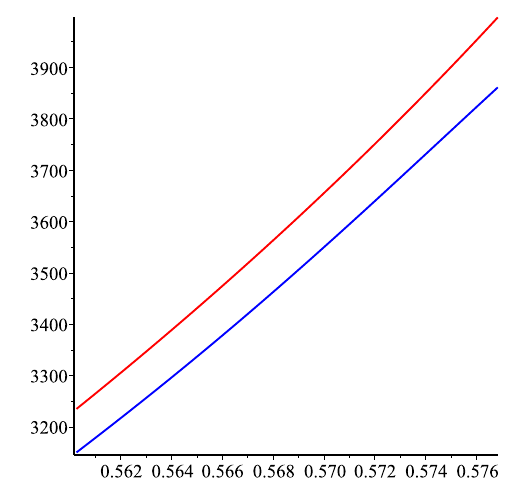}
\caption{First and second level ULB (for Newton potential) and Levenshtein bounds for $(n,\tau)=(12,9)$}
\label{12-9-ulb}
\end{figure}


Even though the two ULB bounds look very close, as numerical comparison in \cite[Table 1] {BDHSS} reveals, the actual energy and the first level ULB are very close to start with. The juxtaposition of the Levenshtein bound and its second level lift illustrates a significant change. In fact the parameters $n=12$ and
$\tau=9$ are chosen to illustrate the improvement on $\mathcal{A}(n,s)$ since in small dimensions it may look insignificant. 

\medskip

In Tables \ref{AC} and \ref{AD} we show how different situations arise in dimension $n=8$ and for $\tau(8,N) \in \{5,6,7,8\}$ and $n=9$ and for $\tau(9,N) \in \{5,6,7,8\}$ respectively. No test functions are 
negative for $\tau(8,N) \leq 4$ and $\tau(9,N) \leq 4$; i.e., the first level bounds are LP-optimal in these cases. In the columns ULB1-LP and ULB2-LP we show the
intervals where the first and second level bounds are LP-optimal, respectively. The columns No ULB2 collect intervals where 
some of the necessary conditions from Section 3 is not satisfied. Finally, the column ULB2 shows intervals where the second level bounds
exist but are not optimal. 

\medskip

More extensive data for $3\leq n\leq 12$ and $3\leq N \leq 1007$ can be found at https://my.vanderbilt.edu/edsaff/.

\medskip

\begin{table}
\caption{First and second level ULB optimality for $n=8$, $90\leq N \leq 990$ .}
\label{AC}

\vspace*{3mm}

\centering

\begin{tabular}{|c|c|c|c|c|c|c|}
\hline
   $\tau(8,N)$ & $[D(8,\tau),D(8,\tau+1)]$ &  ULB1-LP    &     No  ULB2  & ULB2-LP & ULB2 & No ULB2 \\ \hline
  $5$ &  $[72,156]$ &  $[72,79], [91,102]$  &  $[80,90]$   &   &     &        [103,155]      \\ \hline
  $6$ &  $[156,240]$ &    &  $[156,167]$   &  $[168,239]$  &     &            \\ \hline       
 $7$ &  $[240,450]$ &  $240$  &     & $[241,316]$  &     & $[317,449]$             \\ \hline        
  $8$ &  $[450,660]$ &    &     & $[450,532]$  & $[533,625]$     & $[626,659]$             \\ \hline      
 \end{tabular}
\end{table}

\vspace{10mm}

\begin{table}
\caption{First and second level ULB optimality for $n=9$, $90\leq N \leq 990$ .}
\label{AD}

\vspace*{3mm}

\centering

\begin{tabular}{|c|c|c|c|c|c|c|}
\hline
   $\tau(9,N)$ & $[D(9,\tau),D(9,\tau+1)]$ &  ULB1-LP    &     No  ULB2  & ULB2-LP & ULB2 & No ULB2 \\ \hline
  $5$ &  $[90,210]$ &  $[90,141]$  &  $[142,209]$   &   &     &              \\ \hline
  $6$ &  $[210,330]$ &    &  $[210,232]$   & $[233,284]$  &     & $[285,329]$             \\ \hline       
 $7$ &  $[330,660]$ &  $[330,332]$  &  $[333,338]$   & $[339,439]$  &     & $[440,659]$             \\ \hline        
  $8$ &  $[660,990]$ &    &     & $[660,850]$  & $[851,929]$     & $[930,989]$             \\ \hline      
 \end{tabular}
\end{table}

\baselineskip=0.99\normalbaselineskip

\begin{table}
\caption{First and second level ULBs vs. harmonic energy for a sample of cases $\ 3 \leq n \leq 5 \ $.}
\label{AB}

\vspace*{6mm}

\centering
\begin{tabular}{|c|c|c|c|c|c|c|c|c|c|}
\hline
    $n$ &  $N$ & $\tau(n,N)$ & $\alpha_k$ & ULB1 & $\beta_{k+1}$ & ULB2 & $s(C)$ & Energy & $L_\tau(n,\beta_{k+1})$   \\ \hline
\rowcolor{Gray}
{$3^\blacktriangle$} & {$12$} &  $5$ &  $0.44721$ &   $98.33050$ &  $---$ &  $---$ &  $0.44721$ &   $98.33050$ &  $---$   \\ \hline
          $3$ & $13^{-+}$ &  $5$ &  $0.48937$ &  $117.50227$ &  $0.49320$ &  $117.52252$ &  $0.61129$ &  $117.70646$ &  $13.10104$   \\ \hline
          $3$ & $14^{--}$ &  $5$ &  $0.52401$ &  $138.43302$ &  $0.52767$ &  $138.45874$ &  $0.60367$ &  $138.61272$ &  $14.11861$   \\ \hline
          $3$ & $15^{+-}$ &  $5$ &  $0.55223$ &  $161.12063$ &        $ $ &          $ $ &  $0.65309$ &  $161.34048$ &         $ $   \\ \hline
\rowcolor{Gray}
          $3$ & $16^{-+}$ &  $6$ &  $0.57531$ &  $185.56365$ &  $0.57655$ &  $185.57396$ &  $0.65689$ &  $185.82331$ &  $16.04667$   \\ \hline
          $3$ & $17^{-+}$ &  $6$ &  $0.60000$ &  $211.85442$ &  $0.60253$ &  $211.88210$ &  $0.64134$ &  $212.10080$ &  $17.11170$   \\ \hline
          $3$ & $18^{--}$ &  $6$ &  $0.62115$ &  $239.93234$ &  $0.62358$ &  $239.96600$ &  $0.67514$ &  $240.16893$ &  $18.12579$   \\ \hline
          $3$ & $19^{--}$ &  $6$ &  $0.63921$ &  $269.79600$ &  $0.64099$ &  $269.82637$ &  $0.70822$ &  $270.17893$ &  $19.10762$   \\ \hline
\rowcolor{Gray}
          $3$ & $20^{-+}$ &  $7$ &  $0.65465$ &  $301.44437$ &        $ $ &          $ $ &  $0.69348$ &  $301.76313$ &         $ $   \\ \hline
 \hline
\rowcolor{Gray}
          $4$ & $14^{-+}$ &  $4$ &  $0.27429$ &   $98.00000$ &        $ $ &          $ $ &  $0.33921$ &   $98.52459$ &         $ $   \\ \hline
      $4^{*}$ & $15^{-+}$ &  $4$ &  $0.30620$ &  $114.95833$ &  $0.30901$ &  $115.00000$ &  $0.35355$ &  $115.23320$ &  $15.09646$   \\ \hline
      $4^{*}$ & $16^{-+}$ &  $4$ &  $0.33333$ &  $133.33333$ &  $0.33668$ &  $133.39481$ &  $0.43652$ &  $133.89967$ &  $16.13540$   \\ \hline
          $4$ & $17^{-+}$ &  $4$ &  $0.35645$ &  $153.12500$ &  $0.35921$ &  $153.18839$ &  $0.47650$ &  $153.96222$ &  $17.13061$   \\ \hline
          $4$ & $18^{-+}$ &  $4$ &  $0.37627$ &  $174.33333$ &  $0.37792$ &  $174.38060$ &  $0.48480$ &  $175.23235$ &  $18.09042$   \\ \hline
          $4$ & $19^{-+}$ &  $4$ &  $0.39337$ &  $196.95833$ &        $ $ &          $ $ &  $0.48797$ &  $197.90580$ &         $ $   \\ \hline
\rowcolor{Gray}
    $4^{\diamond}$ & $20$ &  $5$ &  $0.40824$ &  $221.00000$ &  $---$ &  $---$ &  $0.44168$ &    $221.59853$ &  $---$   \\ \hline
      $4^{*}$ & $21^{-+}$ &  $5$ &  $0.42720$ &  $246.75000$ &  $0.42895$ &  $246.80226$ &  $0.52049$ &  $247.47325$ &  $21.09675$   \\ \hline
      $4^{*}$ & $22^{-+}$ &  $5$ &  $0.44461$ &  $274.00000$ &  $0.44767$ &  $274.10417$ &  $0.54446$ &  $275.03231$ &  $22.18545$   \\ \hline
      $4^{*}$ & $23^{-+}$ &  $5$ &  $0.46050$ &  $302.75000$ &  $0.46399$ &  $302.88662$ &  $0.57524$ &  $304.08398$ &  $23.23343$   \\ \hline
      $4^{*}$ & $24^{--}$ &  $5$ &  $0.47495$ &  $333.00000$ &  $0.47854$ &  $333.15757$ &  $0.50000$ &  $334.00000$ &  $24.26443$   \\ \hline
          $4$ & $25^{--}$ &  $5$ &  $0.48807$ &  $364.75000$ &        $ $ &          $ $ &  $0.60167$ &  $365.97676$ &         $ $   \\ \hline
          $4$ & $26^{+-}$ &  $5$ &  $0.50000$ &  $398.00000$ &        $ $ &          $ $ &  $0.56449$ &  $399.38498$ &         $ $   \\ \hline
          $4$ & $27^{+-}$ &  $5$ &  $0.51084$ &  $432.75000$ &        $ $ &          $ $ &  $0.64815$ &  $434.30824$ &         $ $   \\ \hline
          $4$ & $28^{+-}$ &  $5$ &  $0.52072$ &  $469.00000$ &        $ $ &          $ $ &  $0.64237$ &  $470.79842$ &         $ $   \\ \hline
          $4$ & $29^{+-}$ &  $5$ &  $0.52973$ &  $506.75000$ &        $ $ &          $ $ &  $0.62694$ &  $508.75066$ &         $ $   \\ \hline
\rowcolor{Gray}
          $4$ & $30^{-+}$ &  $6$ &  $0.53798$ &  $546.00000$ &  $0.53982$ &  $546.12516$ &  $0.63014$ &  $548.37233$ &  $30.18048$   \\ \hline
                           \hline
     \rowcolor{Gray}
    $5^{\diamond}$ & $30$ &  $5$ &  $0.37796$ &  $398.22942$ &        $ $ &          $ $ &  $0.41665$ &  $400.57973$ &         $ $   \\ \hline
          $5$ & $31^{-+}$ &  $5$ &  $0.38810$ &  $429.26411$ &        $ $ &          $ $ &  $0.49636$ &  $431.73992$ &         $ $   \\ \hline
      $5^{*}$ & $32^{-+}$ &  $5$ &  $0.39779$ &  $461.55489$ &  $0.39870$ &  $461.65839$ &  $0.44721$ &  $463.22759$ &  $32.09565$   \\ \hline
      $5^{*}$ & $33^{-+}$ &  $5$ &  $0.40702$ &  $495.10289$ &  $0.40860$ &  $495.29351$ &  $0.52494$ &  $498.25726$ &  $33.17595$   \\ \hline
      $5^{*}$ & $34^{-+}$ &  $5$ &  $0.41580$ &  $529.90910$ &  $0.41781$ &  $530.17012$ &  $0.55380$ &  $533.83563$ &  $34.23704$   \\ \hline
      $5^{*}$ & $35^{-+}$ &  $5$ &  $0.42413$ &  $565.97439$ &  $0.42637$ &  $566.28683$ &  $0.57230$ &  $570.72828$ &  $35.27872$   \\ \hline
      $5^{*}$ & $36^{--}$ &  $5$ &  $0.43202$ &  $603.29953$ &  $0.43436$ &  $603.64803$ &  $0.51722$ &  $607.97487$ &  $36.30722$   \\ \hline
      $5^{*}$ & $37^{--}$ &  $5$ &  $0.43950$ &  $641.88518$ &  $0.44196$ &  $642.26961$ &  $0.57729$ &  $647.27793$ &  $37.34082$   \\ \hline
          $5$ & $38^{--}$ &  $5$ &  $0.44659$ &  $681.73194$ &        $ $ &          $ $ &  $0.56512$ &  $687.15114$ &         $ $   \\ \hline
          $5$ & $39^{+-}$ &  $5$ &  $0.45330$ &  $722.84035$ &        $ $ &          $ $ &  $0.58602$ &  $728.31676$ &         $ $   \\ \hline
          $5$ & $40^{+-}$ &  $5$ &  $0.45965$ &  $765.21089$ &        $ $ &          $ $ &  $0.56248$ &  $769.75044$ &         $ $   \\ \hline
\end{tabular}

\end{table}

In Table \ref{AB} we present examples of bounds for small dimensions and cardinalities which behave typically. For each pair $(n,N)$ 
we give the values of the corresponding parameters $\alpha_k$, $\beta_{k+1}$ and the best known maximal inner product $s(C)$ (i.e., 
an upper bound on $s(n,N)$). The first level and second level ULBs are shown in the fifth and seventh column, respectively, and the
best known Newton energies are shown in the ninth column. The value of the corresponding Levenshtein bound is shown in the last column 
only if it is improved by our second level bound (equal to $N$). The data for eighth and ninth columns is taken from 
 \cite{BBCGKS} (see https://aimath.org/data/paper/BBCGKS2006/). 

The case $(3,12)$ corresponds, of course, to the icosahedron -- a universally optimal code (indicated with superscript $^\blacktriangle$). The empty cells for ULB2 mean that 
some of the necessary conditions from Section 3 is not satisfied (the positive definiteness of the Hermite interpolant 
$H_{\Lambda_{n,k}}(q_kq_{k+1};q_{k+1}^2)$ fails first). The superscripts $^{\diamond}$ and $^{*}$ mean that ULB1 and ULB2 
are LP-optimal, respectively. The superscripts of the cardinalities show the signs of the test-functions $Q_{\tau(n,N)}+3$ and 
$Q_{\tau(n,N)}+4$, respectively. 

\medskip

In conclusion we formulate the following conjecture.

\begin{conjecture}
For every dimension $n$ the values $N\in [D(n,\tau),D(n,\tau+1)]$ for which second level bounds exist form an interval.
\end{conjecture}

\textit{Acknowledgements}{
The research of the first and fifth authors was supported, in part, by a Bulgarian NSF contract DN02/2-2016.
The research of the second author was supported, in part, by a Simons Foundation grant no. 282207.
The research of the third and fourth authors was supported, in part, by the U. S. National Science Foundation under grant DMS-1516400.
The first author is also with Technical Faculty, South-Western University, Blagoevgrad, Bulgaria.
The research for this article was concluded while four of the authors were in residence at Oberwolfach Research Institute for Mathematics 
under the "Research in Pairs" program in 2019.}


\begin{thebibliography}{0}
\bibitem{A1} N.\,N.\,Andreev,
A spherical code,
\emph{Uspekhi Mat. Nauk} 54, (1999): 255-256.

\bibitem{A2} N.\,N.\,Andreev,
A minimal design of order $11$ on the $3$-sphere,
\emph{Math. Notes} 67, (2000): 417-424.

\bibitem{AY} N.\,N.\,Andreev, V.\,A.\,Yudin,
Polynomials of least deviation from zero and Chebyshev-type cubature formulas,
\emph{Proc. Steklov Inst. Math.} 232, (2001): 39-51.

\bibitem{AB} V.\,V.\,Arestov, A.\,G.\,Babenko,
Estimates of the maximal value of angular code distance for $24$ and $25$ points on the unit sphere in ${\mathbb R}^4$,
\emph{Math. Notes} 68, (2000): 419-435.

\bibitem{BBCGKS} B.\,Ballinger, G.\,Blekherman, H.\,Cohn, N.\,Giansiracusa, E.\,Kelly, A.\,Sch\"{u}rmann, 
Experimental Study of Energy-minimizing Point Configurations on Spheres, 
{\em Experiment. Math.} {\bf 18}, 257--283, (2009).

%
\bibitem{BHS} S.\,Borodachov, D.\,Hardin, E.\,Saff, \emph{Discrete Energy on Rectifiable Sets}, Springer, 2019 (to appear)

\bibitem{Bo} K.\,Boroczky Packing of spheres in spaces of constant curvature. \emph{Acta
Math. Hung.}, 32 (1978), 243-261.

\bibitem{BoJr} K.\,Boroczky Jr., \emph{Finite packing and covering}. Cambridge University
Press, 2004.

\bibitem{Boy1} P.\,Boyvalenkov, Extremal polynomials for obtaining bounds
for spherical codes and designs, \emph{Discr. Comp. Geom.} 14, (1995): 167-183.

\bibitem{BDB} P.\,Boyvalenkov, D.\,Danev, S.\,Bumova, 
Upper bounds on the minimum distance of spherical codes, 
\emph{IEEE Trans. Inform. Theory} 41, (1996): 1576-1581.

\bibitem{BDK} P.\,Boyvalenkov, D. Danev, P. Kazakov,
Indexes of spherical codes, DIMACS,
\emph{ Ser. Disc. Math. \& Theor. Comp. Sci.} 56, (2001): 47-57.

\bibitem{BDHSS} P.\,Boyvalenkov, P.\,Dragnev, D.\,Hardin, E.\,Saff, M.\,Stoyanova.
Universal lower bounds for potential energy of spherical codes,
\emph{Constr. Approx.} 44, (2016): 385-415.

\bibitem{CCEK} H.\,Cohn, J.\,Conway, N.\,Elkies,  A.\,Kumar,
The $D_4$ root system is not universally optimal,
\emph{Experiment. Math.} 16, (2007): 313-320.

\bibitem{CK} H.\,Cohn, A.\,Kumar,
Universally optimal distribution of points on spheres,
\emph{J. Amer. Math. Soc.} 20, (2006): 99-148.

\bibitem{CW} H.\,Cohn, J.\,Woo, Three point bounds for energy minimization,
\emph{J. Amer. Math. Soc.} 25, (2012): 929-958.

\bibitem{Co1} H.\,S.\,M.\,Coxeter, \emph{Introduction to Geometry}, 2nd ed. New York: Wiley, 1969.

\bibitem{Co2} H.\,S.\,M.\,Coxeter, \emph{Regular Polytopes}, 3rd. ed., Dover Publications, 1973. ISBN 0-486-61480-8.

\bibitem{D} P.\,J.\,Davis, 
\emph{Interpolation and Approximation},
Blaisdell Publishing Company, New York, 1963.

\bibitem{DGS} P.\,Delsarte, J.-M.\,Goethals, J.\,J.\,Seidel,
Spherical codes and designs, 
\emph{Geom. Dedicata} 6, (1977): 363-388.

\bibitem{DL} P.\,Delsarte, V.\,I.\,Levenshtein, 
Association schemes and coding theory, 
\emph{Trans. Inform. Theory} 44, (1998): 2477-2504.

\bibitem{Gas} G. Gasper, 
Linearization of the product of Jacobi polynomials, II,
\emph{Canad. J. Math.} 22, (1970): 582-593.

\bibitem{Hau78} W.\,Haussmann,
Differentiable Tchebycheff subspaces and Hermite interpolation,
\emph{Annales Acad. Scientiarum Fennicre}, Ser. A. I. Mathematica, 4, (1978): 75-83.

\bibitem{Ike73} Y.\,Ikebe,
Hermite-Birkhoff interpolation problems in Haar subspaces,
\emph{J. Approx. Theory} 8, (1973): 142-149.

\bibitem{KL} G.\,A.\,Kabatiansky, V.\,I.\,Levenshtein,
Bounds for packings on a sphere and in space,
\emph{Probl. Inform. Transm.} 14, (1978): 1-17.

\bibitem{Lev78}  V.\,I.\,Levenshtein,
On bound for packings in $n$-dimensional Euclidean space,
\emph{Soviet Math. Dokl.} 20, (1979): 417-421.

\bibitem{Lev83}  V.\,I.\,Levenshtein,
Bounds for packings of metric spaces and some their applications,
in \emph{Probl. Cybern.} 40, Nauka, Moscow, (1983): 43-110 (in Russian).

\bibitem{Lev92} V.\,I.\,Levenshtein,
Designs as maximum codes in polynomial metric spaces,
\emph{Acta Appl. Math.} 25, (1982): 1-82.

\bibitem{Lev} V.\,I.\,Levenshtein,
Universal bounds for codes and designs,
\emph{Handbook of Coding Theory}, V.\,S.~Pless and W.\,C.~Huffman, Eds., 
Elsevier, Amsterdam, Ch.~6, (1998): 499-648.

\bibitem{MP00} H.\,N.\,Mhaskar, D.\,V.\,Pai,
\emph{Fundamentals of Approximation Theory},
Alpha Science International, 2000.

\bibitem{M} O.\,R.\,Musin,
The kissing number in four dimensions.
\emph{Ann. of Math.} 168, (2008): 1-32.

\bibitem{NN} S. Nikova, V. Nikov, Extremal polynomials for codes in polynomial metric spaces,
in Proceedings of IEEE International Symposium on Information Theory (ISIT2000), Sorrento, Italy, June 25-30, 2000, 60. 

\bibitem{OS} A.\,M.\,Odlyzko, N.\,J.\,A.\,Sloane,
New bounds on the number of unit spheres that can touch a unit sphere in $n$ dimensions,
\emph{J. Comb. Theory} A26, (1979): 210-214.

\bibitem{Si} B.\,Simon,
\emph{Orthogonal Polynomials on the Unit Circle, Part 1: Classical Theory},
American Mathematical Society Colloquium Publications 54,
American Mathematical Society, Providence, RI, 2005.

\bibitem{Sze} G.\,Szeg\H{o},
\emph{Orthogonal Polynomials},
American Mathematical Society Colloquium Publications 23,
American Mathematical Society, Providence, RI, 1939.

\bibitem{Y} V.\,A.\,Yudin,
Minimal potential energy of a point system of charges,
\emph{Discret. Mat.} 4, (1982): 115-121 (in Russian);
English translation: Discr. Math. Appl. 3, (1983): 75-81.

\end{thebibliography}
\end{document}